\documentclass[a4paper,12pt,reqno]{amsart}

\usepackage{float}
\usepackage{graphicx}
\usepackage{tikz}
\usetikzlibrary{calc}
\usepackage{amsmath,amsthm,amssymb,amsfonts}
\usepackage{caption}
\usepackage{subcaption}
\usepackage{latexsym}
\usepackage{color}     
\usepackage{hyperref}
\usepackage{amsmath}                       
\usepackage{amssymb}                       
\usepackage{amsfonts}                      
\usepackage{amsthm}     
\usepackage{enumerate}    
\usepackage{comment}
\newcommand{\say}[1]{``#1''}

\usepackage{enumitem}
\setlist{nolistsep}

\newcommand{\diam}{\operatorname{diam}}
\newcommand{\cE}{\mathcal{E}}
\newcommand{\sign}{\text{sign}}

\newcommand{\epse}{\kappa} 
\newcommand{\epsp}{\varepsilon} 

\newcommand{\RR}{\mathbb R}
\newcommand{\NN}{\mathbb N}

\def\Eins{{\mathchoice {\mathrm{1\mskip-4mu l}} {\mathrm{1\mskip-4mu l}}%
					{\mathrm{1\mskip-4.5mu l}} {\mathrm{1\mskip-5mu l}}}}

\newcommand{\cH}{\mathcal{H}}

\newcommand{\cY}{\mathcal{Y}}
\newcommand{\eps}{\varepsilon}

\newcommand{\cof}{\operatorname{cof}}

\newcommand{\mbf}[1]{\boldsymbol{#1}}
\newcommand{\To}{\longrightarrow}

\newcommand{\abs}[1]{\left\vert#1\right\vert}

\newcommand{\absb}[1]{\big\vert#1\big\vert}

\newcommand{\norm}[1]{\left\|#1\right\|}

\newcommand{\normn}[1]{\|#1\|}
\newcommand{\dist}[2] {\operatorname{dist}\left(#1;#2\right)}
\newcommand{\Dist}[2] {\operatorname{Dist}\left(#1,#2\right)}
\newcommand{\mysetr}[2] {\left\{#1\,\left|\,#2\right.\right\}}
\newcommand{\mysetl}[2] {\left\{\left.#1\,\right|\,#2\right\}}

\newcommand{\identity}{\operatorname{id}}

\newcommand{\idmatrix}{\Eins}

\newcommand{\be}{\begin{equation}}
\newcommand{\ee}{\end{equation}}
\newcommand{\bald}{\begin{aligned}}
\newcommand{\eald}{\end{aligned}}
\newcommand{\baldat}{\begin{alignedat}}
\newcommand{\ealdat}{\end{alignedat}}

\newcommand{\AIB}{\operatorname{AIB}}
\newcommand{\BiLip}{\operatorname{BiLip}}

\newcommand{\dx}{\,\mbox{d}x}
\newcommand{\dz}{\,\mbox{d}z}
\newcommand{\dr}{\,\mbox{d}r}
\newcommand{\ds}{\,\mbox{d}s}
\newcommand{\dt}{\,\mbox{d}t}

\newcommand{\dtx}{\,\mbox{d}\tilde{x}}

\DeclareMathOperator{\tr}{tr}

\newtheorem{theorem}{Theorem}

\newtheorem{thm}[theorem]{Theorem}
\newtheorem{cor}[theorem]{Corollary}
\newtheorem{lem}[theorem]{Lemma}
\newtheorem{prop}[theorem]{Proposition}
\theoremstyle{definition}
\newtheorem{defn}[theorem]{Definition}
\newtheorem{ex}[theorem]{Example}
\theoremstyle{remark}
\newtheorem{rem}[theorem]{Remark}
\renewcommand{\proofname}{\bfseries{Proof}}
\renewenvironment{proof}[1][\proofname]{\par
  \normalfont
  \trivlist
  \item[\hskip\labelsep\itshape
    \bfseries{#1.}]\ignorespaces
}{%
  \qed\endtrivlist
}

\numberwithin{equation}{section}

\begin{document}

\title[Surface penalization of self-interpenetration]{Surface penalization of self-interpenetration in linear and nonlinear elasticity}
\author[Stefan Kr\"omer]{Stefan Kr\"omer}
\address{Stefan Kr\"omer, Institute of Information Theory and Automation, Czech Academy of Sciences, Pod vod\'{a}renskou v\v{e}\v{z}\'{\i}~4, 182~08~Praha~8, Czech Republic, \email{}{skroemer@utia.cas.cz}}
\author[Jan Valdman]{Jan Valdman}
\address{Jan Valdman, Institute of Information Theory and Automation, Czech Academy of Sciences, Pod vod\'{a}renskou v\v{e}\v{z}\'{\i}~4, 182~08~Praha~8, 
Czech Republic and Faculty of Information Technology, Czech Technical University in Prague, Th\'{a}kurova 9, 160~00~Praha~6, Czech Republic,
\email{}{jan.valdman@utia.cas.cz}}
\date{\today}

\begin{abstract}
We analyze a term penalizing surface self-penetration, 
as a soft constraint for models of hyperelastic materials to approximate the Ciarlet-Ne\v{c}as condition (almost everywhere global invertibility of deformations).
For a linear elastic energy subject to an additional local invertibility constraint,
we prove that the penalized elastic functionals converge to the original functional subject to the Ciarlet-Ne\v{c}as condition.
The approach also works for nonlinear models of non-simple materials including a suitable higher order term in the elastic energy, without artificial local constraints.
Numerical experiments illustrate our results for a self-contact problem in 3d.
\end{abstract}


\maketitle

\section{Introduction\label{sec:intro}}

This article is a follow-up of \cite{KroeVa19a}, contributing to an ongoing effort of obtaining a mathematically rigorous computational approach for obtaining steady states or global energy minimizers in models of elastic solids (for general background, see \cite{Sil97B,Cia88B,Ant05B}, e.g.) featuring a global non-penetration constraint reflecting non-interpenetrability of matter. Here, we consider an approximation scheme \say{rigorous} if its solutions 
can be proved to converge in some sense to a solution of the original constrained problem. As we intend to focus on large deformations, we naturally have to study problems of self-contact (frictionless, which is the most simple case). Such problems are always inherently nonconvex, even if a linear elastic model is used for the local response of the material to stresses.
As a consequence, many techniques developed for rigid substrate contact problems, including a formulation of the problem as a variational inequality, fail. 
Even models for one-dimensional rods in 2d already encounter some subtleties \cite{MaPeTo12a}. 
Moreover, results related to Lagrange multiplier theory are only available for elastic models of non-simple materials which offer higher order regularity \cite{PaHea17a,Pa18a}.
The study of critical points is therefore largely out of reach on the analytical level, 
and we are left with the more accessible theory for global energy minimizers as pioneered by Ball for hyperelasticity \cite{Ba77a}. 
On the level of numerical convergence results, we have to handle the possibility of non-uniqueness of solutions, and as a consequence, convergence to one of  them can typically only be ensured for suitable subsequences of the approximate solutions. As we are in a variational framework, we will take advantage of the language of De Giogi's $\Gamma$-convergence to express this, cf.~Remark~\ref{rem:Gamma-convergence} (or the slightly stronger Mosco-convergence).

On the level of the model, a non-penetration constraint translates to global injectivity of the deformation map $y:\Omega\to \RR^d$ mapping the \say{reference configuration} $\Omega\subset \RR^d$ to the deformed state (typically in dimension $d=3$). In suitable spaces of orientation preserving deformations, almost everywhere global injectivity
of $y$ is equivalent to the well-known Ciarlet-Ne\v{c}as condition
\cite{CiaNe87a}
\begin{align}\label{ciarletnecas}
\int_\Omega \abs{\det(\nabla y)} \dx \leq \abs{y(\Omega)}.
\end{align}
As no rigorous and computationally feasible projection onto the constraint \eqref{ciarletnecas} is known, we approach it with a penalty method, roughly following \cite{KroeVa19a}. 
For a class of nonlinear elastic models leading to deformations of bounded distortion, major progress has recently been made in \cite{KroeRei22Pa}, which uses a completely different penalty term which allows successful convergence analysis without any additional local constraints or higher order regularity as in \cite{KroeVa19a}. 

In this paper, we will follow a more practically-minded path, combining a new penalty term supported only on the boundary (or the piece where expected contact) with a linear elastic model. It is introduced  and analyzed  in Section~\ref{sec:pen}.  One important feature of our penalty is that like its bulk variant 
in \cite{KroeVa19a}, it can fully prevent interpenetration as shown in Corollary~\ref{cor:boundaryinvert}. 
Its convergence in combination with elastic energies, also discretized, is discussed in Section~\ref{sec:conv}. 
Compared to \cite{KroeVa19a}, restricting the penalization to the boundary allows a major reduction of computational cost, effectively reducing the dimension by two, as our penalty term is a nonlocal double integral. Here, notice that some degree of nonlocality is unavoidable -- it simply reflects the nonlocal nature of global injectivity.   

 Of course, our choice of a linear elastic model is largely arbitrary. In fact, our analysis does not even exploit it 
beyond the fact that it avoids the determinant singularity of fully nonlinearly elastic energy densities. 
We prefer it as the main example here because we believe that linear elastic models are used in the vast majority of practical computations, especially in 3d, simply due to their much lower computational cost. Besides, in our concrete numerical experiments, this means that the performance gain from a more efficient interpenetration penalty becomes much more noticable.

Using a linear elastic model comes with the caveat that by itself, it is unable to enforce locally orientation preserving deformations, and so we have to supplement it with a 
local\footnote{actually, it is slightly nonlocal, but with arbitrarily short range} locking-type constraint \eqref{idangleconstraint} which implies a local bi-Lipschitz property and thus prevents local loss of injectivity. 
As explained in greater detail in Subsection~\ref{ssec:LE}, we do not expect this local constraint to represent actual material properties; instead, our philosophy here is that if \eqref{idangleconstraint} (with generously chosen constants) is violated or active, 
deformation gradients deviate so far from the identity that
the linear elastic model is too poor an approximation of reality, anyway. In particular, we do not try to enforce \eqref{idangleconstraint} in our numerical experiments, although it can be checked a posteriori. 
It mainly serves as a restriction necessary so that our convergence analysis is valid, see Theorem~\ref{thm:cnsoftboundary} and our main theoretical result, Theorem~\ref{thm:convergenceLE}.  At the same time, as further discussed below, it avoids a Lavrentiev phenomenon in context of finite element approximations. 
For comparison with \cite{KroeVa19a}, we also provide a convergence result featuring a nonlinear elastic energy 
for non-simple materials involving higher order derivatives, see
Theorem~\ref{thm:convergence}.  By results of \cite{HeaKroe09a,KroeVa19a}, this regularized energy 
implicitly enforces a local bi-Lipschitz property of deformations \eqref{biLi} acting as a suitable replacement for \eqref{idangleconstraint}. 

To keep technicalities to a minimum, the theoretical 
part does not involve partial Dirichlet boundary conditions or force terms, although typical examples of the latter are trivial to add. This is further discussed in Subsection~\ref{ssec:remarks}.

We also point out that there is a rather straightforward and rigorous way to obtain a penalty term directly based on~\eqref{ciarletnecas}, see for instance \cite{MieRou16a}: 

\begin{align*}
	E^{\rm penalty}_\eps(y):=\frac{1}{\eps}\Big(\int_\Omega \abs{\det(\nabla y)} \dx -\abs{y(\Omega)}\Big),\quad\text{with $\eps>0$ small}.
\end{align*}

This approach seems to have major computational disadvantages however, including a very high computational cost and a non-smooth character. We are not aware of any practical implementation of it.  In addition, unlike our penalty terms, it vanishes on all injective deformations and thus always leads to a little bit of interpenetration in computations if the approximated solution is in self-contact with nonvanishing surface contact forces. In particular, 
it probably cannot be generalized at all to a more efficient boundary variant, cf.~Remark \ref{rem:CNboundary}. 

 Apart from that,  there are already many numerical approaches based on heuristical arguments, emphasizing performance but lacking a proof of convergence, see for instance \cite{Pa11a,AigLi2013a,BoZaKoRa15a,FaPouRe16a}.
There also is a rich literature for 
energy terms with self-repulsive properties for curves or surfaces, including numerical results, see e.g.~
\cite{BaReiRie18a,BaRei20a,BaRei21a,BlaReiSchiVo22Pa,YuBraSchuKee21a}.
These typically require higher regularity though, at least $C^1$ which is more than we want to impose here, as this can collide with 
a possible Lavrentiev phenomenon \cite{FoHruMi03a}.

Generally, numerical approximation in the presence of constraints related to local invertibility or orientation preservation
has to handled with care 
to avoid possible Lavrentiev phenomenona or related issues
\cite{Ne90a,BaSu92a}. 
In fact, even the density of finite element spaces in sets of admissible orientation-preserving deformations with finite energy is often nontrivial, as illustrated by the still open Ball-Evans problem \cite{Ba02a}. We also refer to \cite{DaPra14a,DePhPra20a,IwOnnRa21a} for some results in dimension $d=2$ and to \cite{CaHeTe18a} for a counterexample
in Sobolev spaces with low integrability. 
To justify our assumptions in Theorem~\ref{thm:convergenceLE} in this regard, we explicitly show the existence of almost conforming finite elements for our constrained linear elastic model in Proposition~\ref{prop:confFE}. In fact, this is the reason for using the constraint \eqref{idangleconstraint} instead of a straightforward local bi-Lipschitz property like \eqref{biLi}.

Our practical numerical experiments are presented in Section~\ref{sec:num}. There, we heavily exploit the linear elastic model using a Schur complement method, effectively pre-solving the linear elastic problem in the interior to transform the whole problem
to one exclusively depending on boundary nodes.
In practice, we can even use a prescribed subset of \say{non-penetration} boundary nodes, if we restrict the penalty energy to the the associated boundary part. In this way, a priori intuition about the expected contact set can be exploited to further reduce the cost of the computation.
%
%
%
%
\section{Penalization terms for the Ciarlet-Ne\v{c}as condition\label{sec:pen}}

The basic idea of a penalty method in variational context is to replace a given constraint by an additional ``penalization'' term in the energy. This term should roughly approximate a functional that vanishes where the constraint holds while it is infinite elsewhere. For asymptotics, it is 
natural to have this approximation governed by a penalization parameter, $\epsp>0$ below, where $\epsp=0$ at least formally corresponds to the case where the original constraint is again perfectly enforced. 

Before we introduce a new penalization term $E^{\partial\Omega}_{\epsp}$ 
 for the Ciarlet-Ne\v{c}as condition  acting only of the surface, let us recall the bulk version introduced in \cite{KroeVa19a}
which serves as a role model. It is given as follows:
\begin{align}\label{sCN1}
\baldat{2}
	&E^{CN}_{\epsp}(y)&&:=
	\frac{1}{\epsp^{\beta}}
	\int_{\Omega} d^{CN}_{\epsp,y}(x)\dx\quad\text{with}\\
	&d^{CN}_{\epsp,y}(x)&&:=\int_{\Omega}
	\frac{1}{\epsp^d} \left[g(\abs{\tilde{x}-x})-g\Big(\frac{\abs{y(\tilde{x})-y(x)}}{\epsp}\Big)\right]^+
	\dtx,
\ealdat
\end{align}
where $[a]^+:=\max\{0,a\}$ denotes the positive part, $\beta>0$ is a constant and
\begin{align}\label{sCN1b}
\bald
	&g:[0,\infty)\to [0,\infty)~~\text{is a continuous,}\\
	&\text{strictly increasing function with $g(0)=0$.} 
\eald
\end{align}
For locally bi-Lipschitz deformations $y$, as $\epsp\to 0$, $E^{\rm CN}_{\epsp}(y)$ converges to zero whenever $y$ satisfies the Ciarlet-Ne\v{c}as condition \eqref{ciarletnecas} and is out of self-contact, and to $+\infty$ if \eqref{ciarletnecas} is violated. Borderline cases with self-contact are more subtle but also lead to the correct limit if $y$ is replaced by an appropriate approximating sequence (a suitable recovery sequence in the spirit of Gamma convergence). For more details see \cite[Thm.~3.3 and Thm. 4.6]{KroeVa19a}.

One of the disadvantages of $E^{\rm CN}_{\epsp}$ is its nonlocal nature and the ensuing computational complexity for numerical evaluation: Typically, a single evaluation of the double integral has a cost of the order of $h^{-2d}$ elementary operations, where $h$ is the grid size and $d$ the dimension of the reference configuration $\Omega$. 
While it is clear that the nonlocal nature of global invertibility will also be reflected in any associated penalty term,
the computational cost can be reduced if we work with integration over the boundary $\partial\Omega$ instead of the full domain, effectively decreasing the dimension.

As we will show, this is possible while retaining the main effect of the penalty term, using the following surface variant:
\begin{align}\label{sCN2}
\baldat{2}
	&E^{\partial\Omega}_{\epsp}(y)&&:=
	\frac{1}{\epsp^\beta}
 \int_{\partial \Omega} d^{\partial\Omega}_{\epsp,y}(x)\,\mbox{d}\cH^{d-1}(x)\quad\text{with}\\
	&d^{\partial\Omega}_{\epsp,y}(x)&&:=\int_{\partial \Omega}
	\frac{1}{\epsp^{d-1}} P\left(g(\abs{\tilde{x}-x})-g\Big(\frac{\abs{y(\tilde{x})-y(x)}}{\epsp}\Big)\right)
	\mbox{d}\cH^{d-1}(\tilde{x}),
\ealdat
\end{align}
where $\beta> 0 $, $g$ satisfies \eqref{sCN1b} as before, $\cH^{d-1}$ denotes the $(d-1)$-dimensional Hausdorff measure (the surface measure)	and
$P$ is an approximation of the positive part $[\cdot]^+$ in the sense that
\begin{align}\label{sCN1c}
\begin{aligned}
	&P:\RR\to [0,\infty)~~\text{is continuous with}~~[t-1]^+\leq P(t) \leq [t]^+,~~\\
	&\text{and}~~P(t)>0 ~~\text{for}~~t>0.
\end{aligned}
\end{align}
In practice, we choose $P$ and $g$ as smooth approximations of the map $t\mapsto [t]^+$. 

Illustrations of $d^{CN}_{\epsp,y}$ and $d^{\partial\Omega}_{\epsp,y}$ for a 2d deformation with self-penetration are given in Figure~\ref{fig:1-4}  and were generated by the complementary code of \cite{KroeVa19a} .  
\begin{figure}[htb]
    \begin{minipage}[t]{.45\textwidth}
        \centering
        \includegraphics[width=\textwidth]{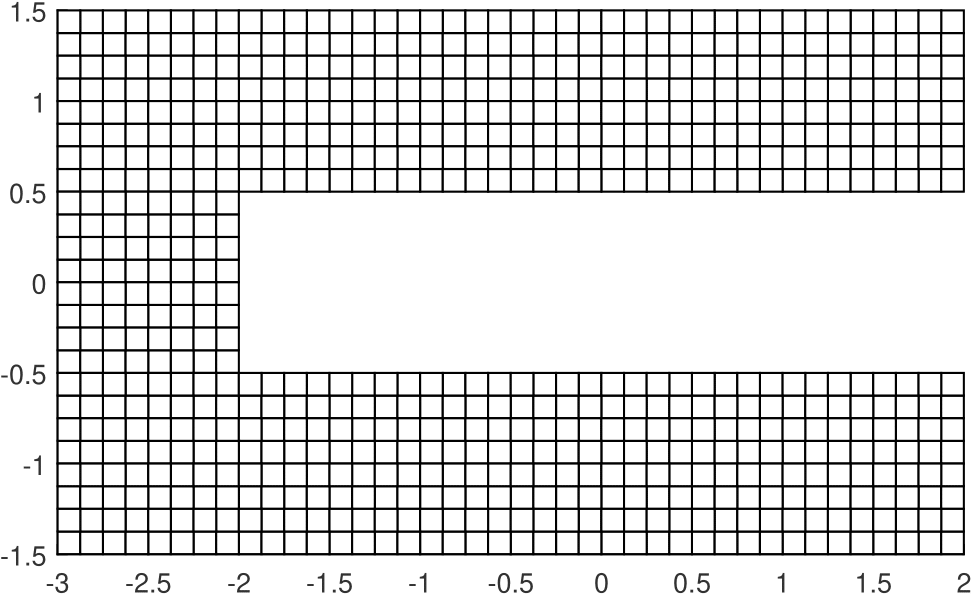}
        \subcaption{Undeformed domain.}
        \vspace{0.5cm}
	  \includegraphics[width=\textwidth]{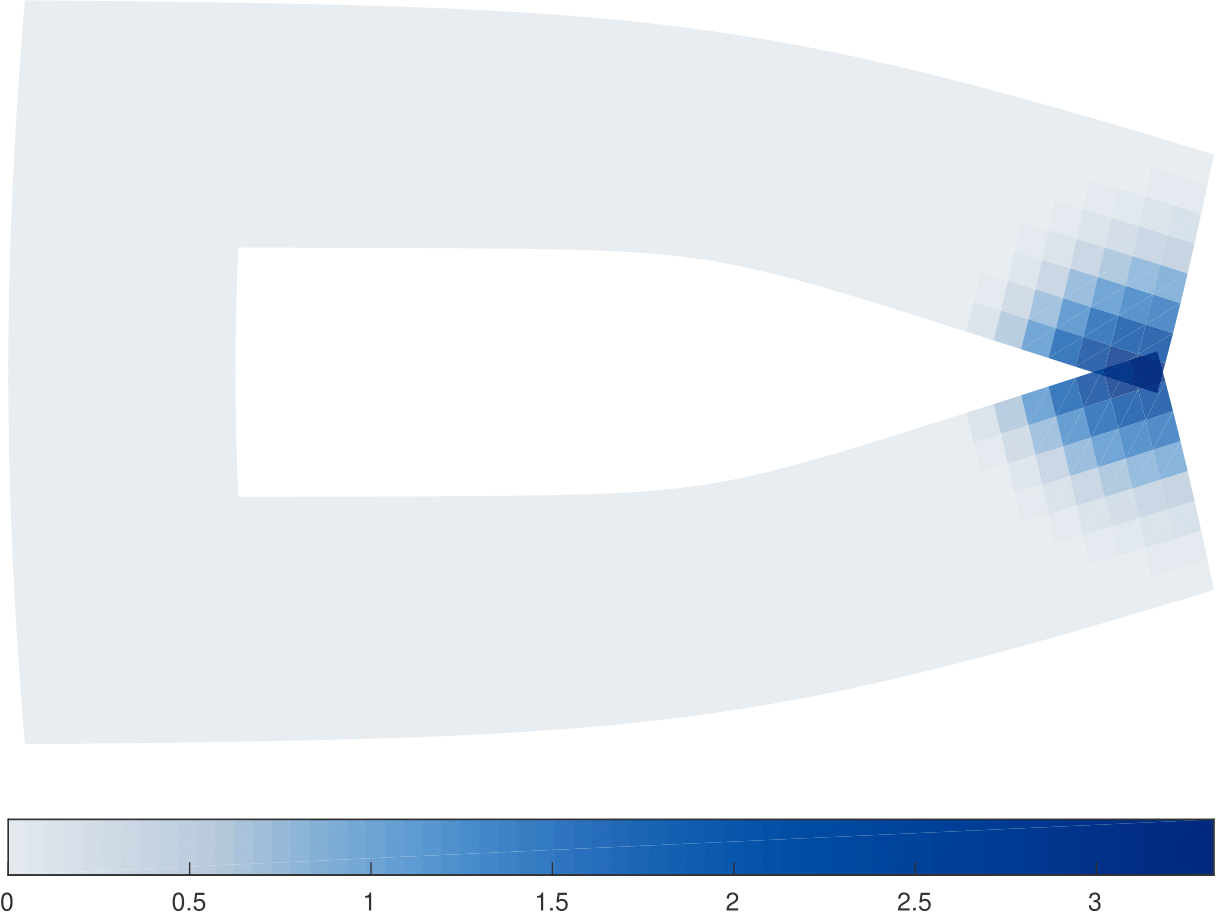}
        \subcaption{Density $d^{CN}_{\epsp,y}(x)$ for $\epsilon=1/2$.}
        \vspace{0.5cm}
	  \includegraphics[width=\textwidth]{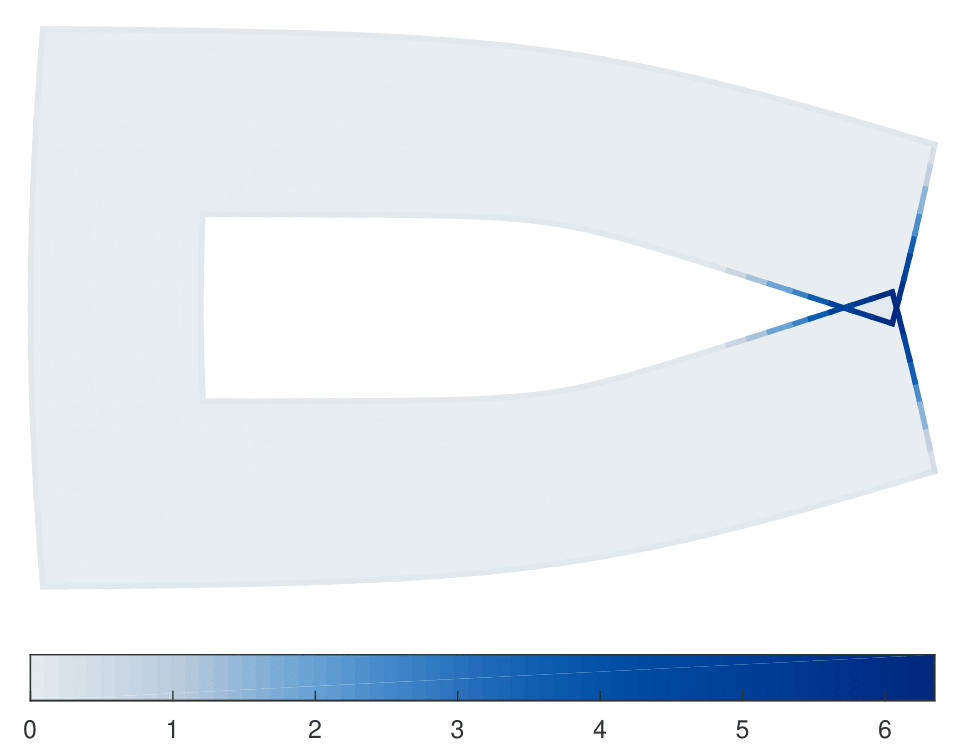}
        \subcaption{Density $d^{\partial\Omega}_{\epsp,y}(x)$ for $\epsilon=1/2$.}
    \end{minipage}
    \hfill
    \begin{minipage}[t]{.45\textwidth}
        \centering
        \includegraphics[width=\textwidth]{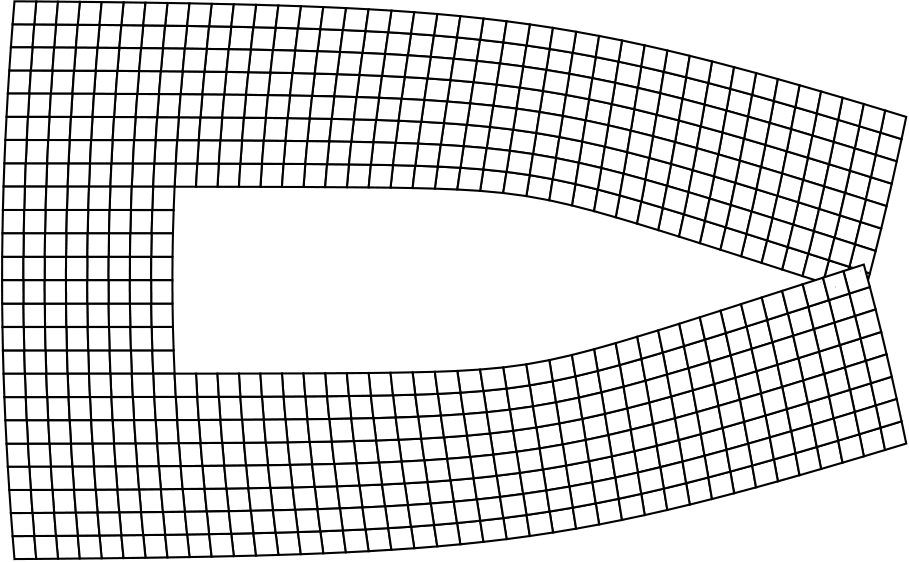}
        \subcaption{Deformed domain.}
				\vspace{0.5cm}			
				\includegraphics[width=\textwidth]{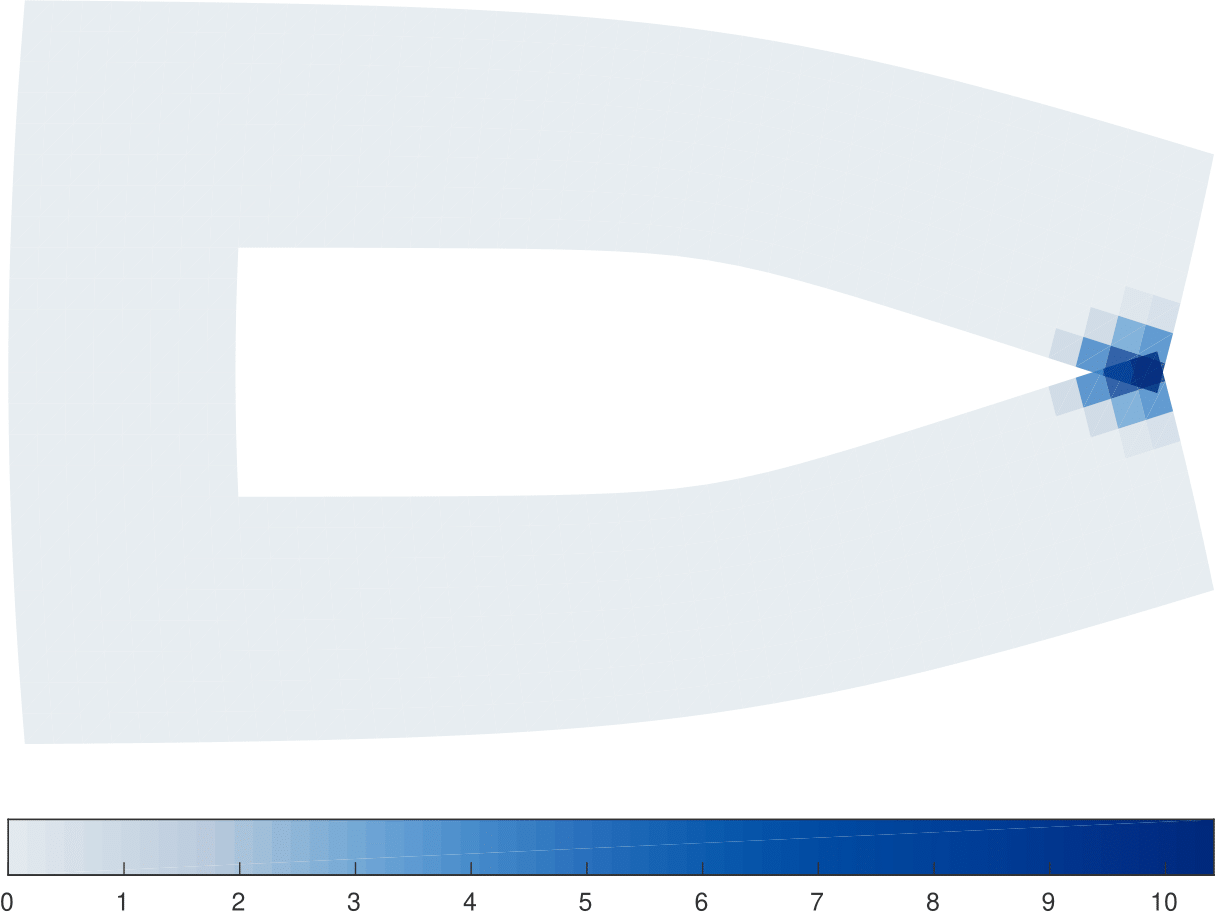}
        \subcaption{Density $d^{CN}_{\epsp,y}(x)$ for $\epsilon=1/4$.}
        \vspace{0.5cm}			
				\includegraphics[width=\textwidth]{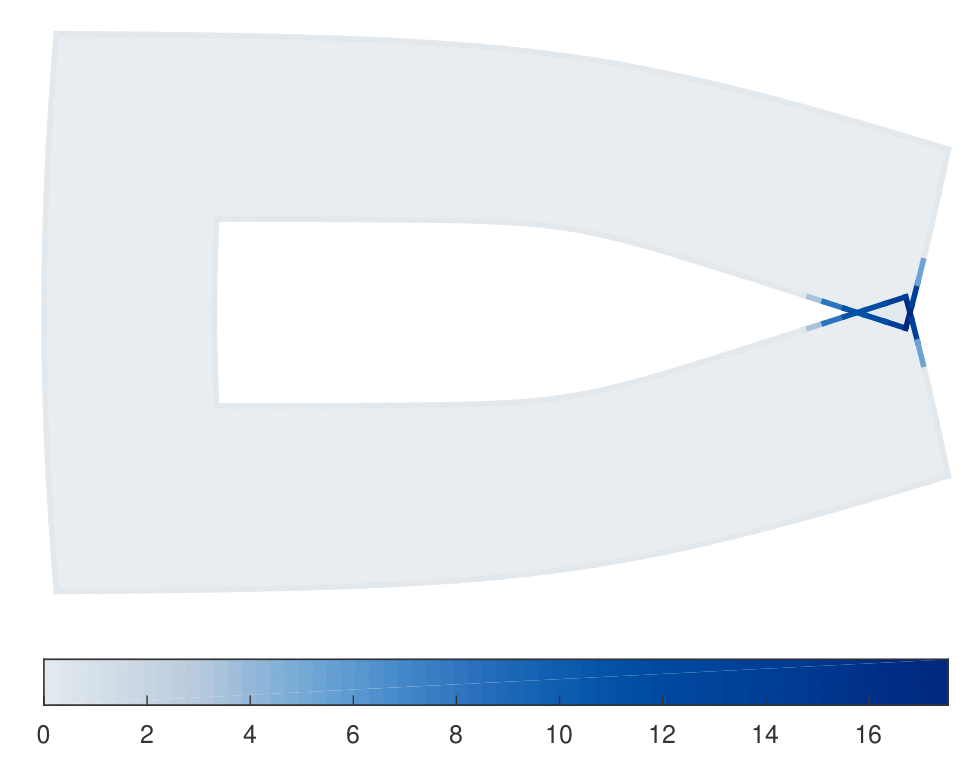}
        \subcaption{Density $d^{\partial\Omega}_{\epsp,y}(x)$ for $\epsilon=1/4$.}
    \end{minipage}  
    \caption{Pincers domain under given deformation.}\label{fig:1-4}
\end{figure}

\subsection{Analytic investigation of the penalty term}

We now analyze the behavior of $E^{\partial\Omega}_{\epsp}$ as $\epsp\to 0$. 

\begin{minipage}{\textwidth}
\begin{thm}[Asymptotics of the surface penalty term \eqref{sCN2}]\label{thm:cnsoftboundary}
~\\
Let $\Omega\subset \RR^d$ be a bounded Lipschitz domain, 
let $E^{\partial\Omega}_{\epsp}$ be the functional defined in \eqref{sCN2} with $\beta>0$, let $g$ satisfy \eqref{sCN1b} and $P$ satisfy \eqref{sCN1c}.  Moreover,  suppose that
$y\in C^{0}(\partial\Omega;\RR^d)$ 
is locally bi-Lipschitz in the sense that
\begin{align}\label{biLiB}
\bald
    l \abs{x_1-x_2} \leq 
	\abs{y(x_1)-y(x_2)} \leq L \abs{x_1-x_2}&\\
	\text{for all $x_1,x_2\in \partial\Omega$ with $|x_1-x_2|<\varrho$}&,
\eald
\end{align}
with some constants $\varrho,l,L>0$.  Finally, for $s\geq 0$, define the set 
\[
	P_{y}(s):=\mysetr{x\in \partial\Omega}{
	\exists \tilde{x}\in \partial\Omega:
	\abs{y(x)-y(\tilde{x})}\leq s~~\text{and}~~\abs{x-\tilde{x}}>\tfrac{\varrho}{2}
	}.
\] 
Then there exist
constants $r, R, a,A,\bar{\eps}>0$ only depending on $d$, $\Omega$, $\varrho$, $l$, $L$,  $P$  and $g$ such that
\begin{align}
	a \cH^{d-1}(P_y(r\epsp)) \leq 
	\epsp^\beta E^{\partial\Omega}_{\epsp}(y)
	\leq A \cH^{d-1}(P_y(R\epsp))
	\label{CN2bounds}
\end{align}
every $0<\epsp\leq \bar{\eps}$.
\end{thm}
\end{minipage}
\begin{rem}
As a consequence of \eqref{biLiB},
\[
	P_y(0)=\{ x\in\partial\Omega \mid \# y^{-1}(y(x))>1\},
\] 
i.e., it is precisely the set where $y|_{\partial\Omega}$ fails to be injective.
\end{rem}
\begin{proof}[Proof of Theorem~\ref{thm:cnsoftboundary}]
 In essence, the proof is analogous to the one  of \cite[Theorem 3.3]{KroeVa19a}. 
 We here present it self-contained for the reader's convenience.  

The crucial quantity to study is the (rescaled) density of the penalty term given by
\begin{align*}
\bald
J_{y,\epsp}(x)&:=\epsp^{d-1} d^{\partial\Omega}_{\epsp,y}(x) \\
&= \int_{\partial\Omega} 
P\left(g(\abs{\tilde{x}-x})-g\Big(\frac{\abs{y(\tilde{x})-y(x)}}{\epsp}\Big)\right)
	\,\mbox{d}\cH^{d-1}(\tilde{x}).
\eald
\end{align*}
Due to the properties of $g$ and $P$ required in \eqref{sCN1b} and \eqref{sCN1c}, for given $x\in \partial\Omega$, $J_{y,\epsp}(x)$ gives a non-zero contribution if and only if there is a point $\tilde{x}$ 
such that
\begin{align}\label{ispenalized}
&\bald
\epsp\abs{\tilde{x}-x}>\abs{y(\tilde{x})-y(x)}.
\eald
\end{align}
(By \eqref{biLiB}, this automatically means there is a set of positive measure of such points.)
Due to the lower bound in \eqref{biLiB}, if $\epsp$ is small enough, \eqref{ispenalized} never happens locally with $\abs{\tilde{x}-x}<\varrho$.

\noindent{\bf Lower bound in \eqref{CN2bounds}:} 
Defining a small enough constant $r$ based on the constants appearing in \eqref{biLiB}, more precisely,
\[
  r:=\min\left\{\tfrac{\varrho}{4},1\right\},
\]
the lower bound in \eqref{CN2bounds} relies on the following observation:
If $x\in P_y(r\epsp)$ and $\tilde{x}_0$ is an admissible choice in the definition of $P_y(r\epsp)$, i.e.,
\begin{align}\label{Py-tildex}
	\text{$\abs{x-\tilde{x}_0}>\frac{\varrho}{2}$ and $\abs{y(x)-y(\tilde{x}_0)}\leq r\epsp$}, 
\end{align}
then \eqref{ispenalized} holds, $\abs{x-\tilde{x}_0}\geq \varrho$
and for $\tilde{x}=\tilde{x}_0$,
\begin{align}\label{lb-pendensity}
	g(\abs{\tilde{x}-x})-g\Big(\frac{\abs{y(\tilde{x})-y(x)}}{\epsp}\Big)\geq c_\varrho >0,
\end{align}
with $c_\varrho:=g\big(\varrho\big)-g\big(\frac{\varrho}{4}\big)$.
In fact, to obtain \eqref{lb-pendensity} it is enough if $r$ is replaced by the bigger constant $\frac{\varrho}{4}$ in \eqref{Py-tildex}.
Combining this with \eqref{biLiB}, we see that
for the smaller constant 
$\tilde{c}_\varrho:=g\big(\frac{3}{4}\varrho\big)-g\big(\frac{\varrho}{2}\big)$,
the set 
\[
	G(x)=G(x,\tilde{x}_0):=\mysetl{\tilde{x}\in \partial\Omega\cap B_{\frac{\varrho}{4}}(\tilde{x}_0)}{
	\text{\eqref{lb-pendensity} holds with $\tilde{c}_\varrho$}}
\]
contains 
all $\tilde{x}\in\partial\Omega$ with $\abs{\tilde{x}-\tilde{x}_0}<\frac{r}{L+1}\epsp$ 
(as long as $\frac{r}{L+1}\epsp\leq \varrho$).
Since $\partial\Omega$ is Lipschitz, this implies that $G(x)$ has a $(d-1)$-dimensional measure of the order of $\epsp^{d-1}$ (or more) for small enough $\epsp$, independently of $x$ and the choice of $\tilde{x}_0$. By \eqref{lb-pendensity} and \eqref{sCN1c}, we conclude that for all $x\in P_y(r\epsp)$, 
$\epsp^{-(d-1)}J_{y,\epsp}(x)$ is bounded from below by a fixed constant, which gives the lower bound in \eqref{CN2bounds}. 

 \noindent{\bf Upper bound in \eqref{CN2bounds}:} 
We exploit  that the integrand of $J_{y,\epsp}(x)$
is always bounded from above:
\[
	g(\abs{\tilde{x}-x})-g\Big(\frac{\abs{y(\tilde{x})-y(x)}}{\epsp}\Big)\leq g(\abs{\tilde{x}-x}) \leq g(\diam \Omega),
\]
 where $\diam \Omega:=\sup_{x_1,x_2\in \Omega} \abs{x_1-x_2}$. 
 
The integrand of $J_{y,\epsp}(x)$ vanishes if \eqref{ispenalized} does not hold,
in particular if $\eps<l$ and $\abs{\tilde{x}-x}<\varrho$ due to lower bound in \eqref{biLiB}.
Hence, we obtain that
\begin{align}\label{Jy-ub}
\begin{aligned}
	&J_{y,\epsp}(x)\leq C_1 \cH^{d-1}(Q_y(\eps,x))
\end{aligned}
\end{align}
with
\begin{align*}
\begin{aligned}
Q_y(\eps,x):=\mysetr{\tilde{x}\in\partial\Omega}{\abs{y(\tilde{x})-y(x)}< \eps \diam \Omega~~\text{and}~~
	\abs{\tilde{x}-x}\geq \frac{\varrho}{2}}
\end{aligned}
\end{align*}
and the constant $C_1:=\sup_{0\leq t\leq g(\diam \Omega)} P(t)$.
As another consequence of the lower bound in \eqref{biLiB},
we know that for each $\tilde{x}\in Q_y(\eps,\varrho,x)$,
\[
	z\notin Q_y(\eps,x)\quad\text{for all $z\in \partial\Omega$ with}~~\varrho\geq \abs{z-\tilde{x}}\geq \frac{2 \diam\Omega}{l} \eps. 
\]
This means that $Q_y(\epsp,x)$ can be covered by a finite number (depending only on $\varrho$ and $\Omega$) of balls of radius $\frac{2 \diam\Omega}{l} \eps$. Consequently,
\[
	\cH^{d-1}\big(Q_y\big(\epsp,x\big)\big)\leq C_2 \epsp^{d-1},
\]
where the constant $C_2=C_2(\varrho,l,\Omega)>0$ also compensates the fact the $\partial\Omega$ locally is Lipschitz but not necessarily flat.
With $R:=\diam \Omega$ and $A:=C_1C_2>0$, we conclude that
\[
\begin{aligned}
	A\cH^{d-1}(P_y(R\epsp)) &=
	\int_{\mysetr{x\in \partial\Omega}{Q_y(\epsp,x)\neq \emptyset}} C_1C_2 \,\mbox{d}\cH^{d-1}(x) \\
	&\geq \int_{\partial\Omega} \frac{C_1}{\epsp^{d-1}}\cH^{d-1}\big(Q_y(\epsp,x)\big) \, \mbox{d}\cH^{d-1}(x).
\end{aligned}
\]
Together with \eqref{Jy-ub}, this yields the upper bound in \eqref{CN2bounds}. 
\end{proof}
The lower bound in \eqref{CN2bounds} immediately implies that if a deformation $y$ has bounded energy including the penalty term for all small $\epsp>0$, then
$\cH^{d-1}(P_y(0))=0$. With a more refined argument, we can even obtain that $P_y(0)=\emptyset$, i.e., full invertibility of $y|_{\partial\Omega}$, as long as $\beta>d-1$:
\begin{cor}[Boundary invertibility for finite penalization]\label{cor:boundaryinvert}
Suppose that the assumptions of Theorem~\ref{thm:cnsoftboundary} hold and let $C>0$.
If $\beta>d-1$ in \eqref{sCN2},
then there exists a constant $0<\tilde{\eps}\leq \bar\eps$ which only depends on $\beta$, $C$, $d$, $\Omega$, $\varrho$, $l$, $L$ and $g$,
such that for all $\epsp<\tilde{\eps}$ and all $y\in C^{0}(\partial\Omega;\RR^d)$ satisfying \eqref{biLiB}, 
\begin{align}\label{yECN-bound}
	E^{\partial\Omega}_{\epsp}(y)\leq C\quad\text{implies that $y|_{\partial\Omega}$ is injective.}
\end{align}
\end{cor}
\begin{proof}[Proof]
Suppose that $y|_{\partial\Omega}$ is not injective. 
Then there is a pair of points $x_1,x_2\in \partial\Omega$ with $x_1\neq x_2$ and $y(x_1)=y(x_2)$.
Since $y$ is locally bi-Lipschitz by \eqref{biLiB}, there are neighborhoods of each of these two points in $\partial\Omega$ which are fully contained in $P_y(r\epsp)$, 
and each neighborhood locally covers $\partial\Omega$ around $x_i$ for a radius of the order of $\epsp$ (or more).
The surface measure of such a neighborhood is therefore of the order of $\epsp^{d-1}$ (or more). Hence, up to a positive multiplicative constant,
$\cH^{d-1}(P_y(r\epsp))\geq \epsp^{d-1}$. Due to the lower bound in \eqref{CN2bounds}, this contradicts our given energy bound, the premise of \eqref{yECN-bound}.
\end{proof}
\begin{rem}\label{rem:CNboundary}
Unlike the bulk variant \eqref{sCN1}, \eqref{sCN2} cannot be expected to correctly reproduce 
global invertibility as a limiting condition as $\epsp\to 0$ for all $\beta>0$. Moreover, 
on the boundary, it is not helpful to have global invertibility merely a.e.. 
Roughly speaking, this is related to the same issue which prevents a straightforward modification of the Ciarlet-Ne\v{c}as condition \eqref{ciarletnecas}
to surface integrals. Such a formally analogous variant of \eqref{ciarletnecas} on the boundary is given by
\begin{align} \label{ciarletnecasboundary}
\int_{\partial\Omega} \big[\det\big(n\otimes n+(\nabla_t y)^\top \nabla_t y\big)\big]^{\frac{1}{2}}\,\mbox{d}\cH^{d-1}(x) = \cH^{d-1}\big(y(\partial\Omega)\big),
\end{align}
where $\nabla_t$ denotes the tangential gradient, i.e., $\nabla_t y=(\nabla y) P_t$ with the orthogonal projection $P_t=\idmatrix-n\otimes n$ onto the tangent space 
orthogonal to the outer normal $n$. 
However, \eqref{ciarletnecasboundary} only rules out self-intersections of the boundary on $(d-1)$-dimensional sets, which is practically useless because
generic self-intersection on the boundary happens on sets of dimension $d-2$, sets of measure zero with respect to the surface measure $\cH^{d-1}$.
\end{rem}

Theorem~\ref{thm:cnsoftboundary} also shows that the penalty term eventually vanishes for deformations without self-contact on the boundary:
\begin{cor}\label{cor:whenECNvanishes}
In the situation of Theorem~\ref{thm:cnsoftboundary}, let $\epsp\leq \bar{\eps}$ and suppose in addition that $y$ is more than a distance of $R\epsp$ away from any non-local self-contact, i.e.,
\begin{align}\label{corwECNv-1}
	\abs{y(x_1)-y(x_2)}> R\epsp\quad\text{for all}~~x_1,x_2\in\partial\Omega~~\text{with}~~\abs{x_1-x_2}> \tfrac{\varrho}{2}
\end{align}
with $R=\operatorname{diam}\Omega$.
Then $E^{\partial\Omega}_{\epsp}(y)=0$. 
\end{cor}
\begin{proof}
This is a direct consequence of the upper bound in \eqref{CN2bounds} and the definition of $P_y(R\epsp)$: \eqref{corwECNv-1} implies that $P_y(R\epsp)=\emptyset$.
\end{proof}


\subsection{On the assumed local bi-Lipschitz property}

For the preceding analysis, deformations are only admissible if they satisfy the local bi-Lipschitz property \eqref{biLiB} on the boundary. While this does restrict the applicability of our results, \eqref{biLiB} can be justified in certain realistic scenarios, in fact even the more restrictive \eqref{biLi} on all of $\Omega$.

First, as explained below, one can study linear elastic models while imposing \eqref{idangleconstraint} or \eqref{biLi} as a constraint, since both are stable under pointwise convergence of $y$. The drawback is that solutions obtained in such a framework are clearly not physical if the constraint is active. However, our theory does not require any particular choice for the local bi-Lipschitz constants $l,L$, 
so that we can admit any fixed deformation $y$ which is not too far from the identity in $W^{1,\infty}$. 
These are the deformations for which a small strain (linear elastic) model is most likely to provide decent approximations. 
Numerically, both can be checked a posteriori relatively easily as long as boundary self-contact is prevented, see Remark~\ref{rem:biLi-apost}. Still, it would be nice to weaken  \eqref{biLiB}  if possible, because Linear Elasticity can be rigorously justified as a Gamma-limit of suitable nonlinear models where the deformation gradients are required to be close to the identity matrix only in, say, $L^2$  \cite{DalNePe02,AgDalDeS12a}  (not in $L^\infty$ as suggested by \eqref{idangleconstraint}).  
For further related results and a discussion of the case without Dirichlet boundary conditions see Remark~\ref{rem:puretraction}. 

Secondly, for nonlinear elastic energies resisting extreme compression, another justification of \eqref{biLi} can be found if suitable higher order terms are included in the nonlinear elastic energy (non-simple materials): A version of the Inverse Mapping Theorem, summarized in Lemma~\ref{lem:biLi} below, then guarantees \eqref{biLi}. For this argument, the extra regularity of deformations provided by the second gradient term in $E_{\epsp,\sigma}$ in Subsection~\ref{ssec:NLEnonsimple} below, together with
 suitable properties of the nonlinear elastic energy density, cf.~\eqref{W1} and \eqref{pqsd}, 
guarantees that in any set of states with bounded energy, $J:=\det\nabla y$ is uniformly positive, 
bounded away from zero  by a result of \cite{HeaKroe09a}:
\begin{lem}[cf.~Lemma 4.1 in \cite{HeaKroe09a}]\label{lem:HK}
Let $\Omega\subset \RR^d$ be a bounded Lipschitz domain, $\alpha>0$ and $q\geq d/\alpha$.
Then for every $C_1,C_2>0$ there exists a constant $\delta=\delta(C_1,C_2,\Omega,\alpha,q,d)>0$ such that
\[
	\int_\Omega J^{-q} \, \dx \leq C_1 \quad\text{implies that}\quad J\geq \delta>0~~~\text{on $\Omega$}
\]
for any $J\in C^\alpha(\Omega)$ with $\norm{J}_{C^\alpha(\Omega)}\leq C_2$ and $J>0$ a.e.~in $\Omega$.
\end{lem}
In particular, if $y\in W^{2,s}$ with $s>d$, we have
that $y\in C^{1,\alpha}$ by embedding, where $\alpha=(s-d)/s$. Consequently, 
$J:=\det \nabla y\in C^\alpha$, and we can apply Lemma~\ref{lem:HK} as long as 
$q\geq d/\alpha=sd/(s-d)$, which is assumed in \eqref{pqsd}.
In this way, \eqref{yconstants}  in the following lemma can be  justified. 
\begin{lem}[Lemma 3.6 in \cite{KroeVa19a}]\label{lem:biLi}
Let $\Omega\subset \RR^d$ be a bounded Lipschitz domain with local Lipschitz constants bounded by a fixed $L_\Omega>0$, let $\alpha,\delta,M_1,M_2>0$
and let $y\in C^{1,\alpha}(\Omega;\RR^d)$ such that 
\begin{align}\label{yconstants}
	\text{$\det \nabla y\geq \delta>0$}
	~~~\text{and}~~~
	\text{$\abs{\nabla y}\leq M_1$}~~~\text{on $\Omega$}~~~\text{and}
~~~\norm{\nabla y}_{C^\alpha(\Omega)}\leq M_2.
\end{align}
Then there exists a $\varrho>0$ which only depends on $\delta,M_1,M_2,\alpha$ and $\Omega$ 
such that for every $\bar{x}\in \RR^d$, $y$ is injective on 
$\Omega(\bar{x},\frac{\varrho}{2}):=B_{\varrho/2}(\bar{x})\cap \Omega$. Moreover, 
$y$ is bi-Lipschitz on $\Omega(\bar{x},\frac{\varrho}{2})$ for any $\bar{x}$,
i.e., $y$ satisfies \eqref{biLi} on $\Omega$ (and thus on $\overline{\Omega}$ by continuity),
where the constants $l,L>0$ can be explicitly chosen as
\[
    l:=\frac{1}{2} \frac{\delta}{M_1^{d-1}},\quad L:=M_1 \sqrt{L_\Omega^2+1}.
\]
\end{lem}

\subsection{From invertibility on the boundary to invertibility everywhere}
It is not difficult to construct examples of deformations exhibiting interior self-penetration while preserving invertibility at the boundary. 
However, by a result of \cite{Kroe20a} generalizing \cite{Ba81a} (see also \cite{HeMoCoOl21a} for another recent extension), 
for topologically simple domains no such example
is possible in the class of locally orientation preserving deformations. For the precise statement given in Theorem~\ref{thm:AIB-CN} below, we recall the concept of deformations that are \emph{approximately invertible on the boundary} used in \cite{Kroe20a}:
\begin{defn}[$\AIB$]\label{def:AIB}
Let $\Omega\subset \RR^d$ be open and bounded, and $y:\bar\Omega \to \RR^d$ with $y\in C(\partial\Omega;\RR^d)$. 
We say that $y$ is \emph{approximately invertible on the boundary}, or, shortly, $y\in \AIB$, if 
there exists a sequence $(\varphi_k)_{k\in\NN}\subset C(\partial\Omega;\RR^d)$ such that $\varphi_k:\partial\Omega\to \varphi_k(\partial\Omega)$ is invertible for each $k$ and
$\varphi_k\to y$ uniformly on $\partial\Omega$.
\end{defn}

\begin{thm}[{\cite[Corollary 6.5 and Remark 6.3]{Kroe20a}}]\label{thm:AIB-CN}
Let $\Omega\subset \RR^d$ be a bounded bounded Lipschitz domain 
such that $\RR^d\setminus \partial\Omega$ has exactly two connected components.
If $p>d$ and $y\in W^{1,p}(\Omega;\RR^d)\cap \AIB$ with $\det\nabla y>0$ a.e.~in $\Omega$, 
then $y$ satisfies the Ciarlet-Ne\v{c}as condition \eqref{ciarletnecas}.
\end{thm}

\begin{rem}\label{rem:AIB-CN}
If $\Omega\subset \RR^d$ is a bounded Lipschitz domain, $p>d$ and $(y_k)\subset W^{1,p}(\Omega;\RR^d)$ is a bounded sequence,
we automatically have that $y_k\in C(\bar\Omega;\RR^d)$ by embedding.
If, in addition, $\varphi_k:=y_k|_{\partial\Omega}$ is invertible for each $k$, for example due to a bound on the surface penalty term 
as in Corollary~\ref{cor:boundaryinvert}, then any limit $y$ of a subsequence weakly converging in $W^{1,p}$ belongs to $\AIB$.
\end{rem}

\begin{rem}
A variant of Theorem~\ref{thm:AIB-CN}, only valid for piecewise affine maps on a tetrahedral mesh which are invertible on the boundary, is given by
\cite[Thm.~3]{AigLi2013a}. An important step in its proof, namely the calculation of the global degree of the deformation map, is not well explained in \cite{AigLi2013a}, though. 
In the more general framework of Theorem~\ref{thm:AIB-CN}, the additional topological assumption on the reference configuration $\Omega$ ($M$ in \cite{AigLi2013a})
plays a crucial role at exactly this point (cf.~\cite[Thm.~4.2 and Appendix B]{Kroe20a}), 
in combination with the orientation preserving property of the full deformation (not just on the boundary). In particular, for the results of \cite{Kroe20a} it is not enough to have the deformed boundary $y(\partial\Omega)$ ($\Psi(\partial M)$ in \cite{AigLi2013a})
coincide with the boundary of some domain (in fact, the latter does not even imply that the reference configuration is connected). To what extent the restriction to piecewise affine maps 
can be exploited to avoid topological assumptions on the reference configuration is not clear to us from the proof in \cite{AigLi2013a}. 
\end{rem}

\section{Convergence of energies\label{sec:conv}}
We will show here that in the limit as the penalty parameter $\epsp$ converges to zero, 
elastic energies augmented with the surface penalty term will reproduce the original energy with 
global injectivity added as a constraint in form of the Ciarlet-Ne\v{c}as condition \eqref{ciarletnecas}. 
As mentioned before,  due to the assumptions of Theorem~\ref{thm:cnsoftboundary}, 
our analysis is limited to cases where deformations are known to be locally bi-Lipschitz. 
We discuss two such scenarios below,  first tailored to linear elastic energies 
subject to a local constraint (Theorem~\ref{thm:convergenceLE}), and then 
to an unconstrained nonlinear elastic energy containing a regularizing term of higher order
(Theorem~\ref{thm:convergence}).

\subsection{Auxiliary results: domain shrinking}
In  both scenarios, we will need  the following two technical lemmas which 
were also implicitly used in \cite{KroeVa19a} and play a similar role in \cite{KroeRei22Pa}. 
For the proofs of the theorems below, the case $\Gamma=\emptyset$ in Lemma~\ref{lem:shrink} and Lemma~\ref{lem:shrinkcont} is enough, 
but we prefer to present a slightly more general version here which is designed to handle additional Dirichlet boundary conditions on $\Gamma$, 
cf.~Remark~\ref{rem:DC}. 

\begin{lem}[domain shrinking]\label{lem:shrink}
Let $\Omega\subset \RR^d$ be a bounded Lipschitz domain, and $\Gamma\subset \partial\Omega$ a closed set. Then there exist a decreasing sequence of closed sets $\Gamma_j$ with $\Gamma\subset \Gamma_{j+1}\subset \Gamma_j \subset \partial\Omega$ 
such that $\bigcap_{j\in \NN} \Gamma_j=\Gamma$ and a sequence of
$C^\infty$-diffeomorphisms 
\[
    \Psi_j:\overline{\Omega}\to \Psi_j(\overline{\Omega})\subset\subset \Omega \cup \Gamma_j
		\quad\text{with}\quad \Psi_j|_\Gamma=\identity|_\Gamma
\]
such that as $j\to\infty$, $\Psi_j\to \identity$ in $C^m(\overline{\Omega};\RR^d)$ for all $m\in \NN$.
If $\Gamma=\emptyset$, we may choose $\Gamma_j=\emptyset$, too.
\end{lem}
\begin{lem}[composition with domain shrinking is continuous]\label{lem:shrinkcont}~\\
Let $\Omega\subset \RR^d$ be a bounded Lipschitz domain, $k\in \NN_0$, $1\leq p <\infty$ and $f\in W^{k,p}(\Omega;\RR^n)$.
With the maps $\Psi_j$ of Lemma~\ref{lem:shrink}, we then have that $f\circ \Psi_j\to f$ in $W^{k,p}(\Omega;\RR^n)$.
\end{lem}
\begin{proof}[Proof of Lemma~\ref{lem:shrink}]
If $\Omega$ is strictly star-shaped with respect to a point $x_0\in\Omega$ and $\Gamma=\emptyset$, one may take 
$\Psi_j(x):=x_0+\frac{j-1}{j}(x-x_0)$. 
For the general case, we combine local constructions near the boundary using a smooth decomposition of unity: 
First, choose functions
\[
	\text{$\eta_i \in C^\infty(\RR^d;[0,1])$ such that }\eta_i(x)=\begin{cases} 
		0 & \text{if $x\in \Gamma$}, \\
		1 & \text{if $\dist{x}{\Gamma}\geq \tfrac{1}{i}$},
	\end{cases}
\]
(with the understanding that $\eta_i\equiv 1$ if $\Gamma=\emptyset$)
and define 
\[
	\Gamma_i:=\mysetr{x\in\partial \Omega}{\eta_{n}(x)=0~~\text{for}~~n=1,\ldots,i}.
\]
If, locally in some open cube $Q$, $\Omega$ is a Lipschitz subgraph, i.e.,
\[	
	\Omega\cap Q=\{x\in Q\mid x\cdot e<f(x')\}~\text{and}~\partial\Omega\cap Q=\{x'+ef(x')\mid x\in Q\},
\]
where $e$ is a unit vector orthogonal to one of the faces of $Q$, $x':=x-(x\cdot e)e$ and $f$ is a real-valued Lipschitz function,
we define for $x\in Q$
\[
\begin{aligned}
    \hat{\Psi}_j(x;Q):=(1-\eta_{i(j)}(x))x+\eta_{i(j)}(x)\big(x'+\alpha_0 e+\frac{j-1}{j}(e\cdot x-\alpha_0)e\Big),
		&\\
		\text{where $\alpha_0:=\inf_{x\in Q}e\cdot x$ and $i(j)\in \NN$ will be fixed later}.&
\end{aligned}
\]
Notice that $\hat{\Psi}_j(\cdot;Q)$ keeps $\Gamma$ fixed, but as soon as 
as we are more than $\tfrac{1}{i(j)}$ away from $\Gamma$, it pulls the local boundary piece $\partial\Omega\cap Q$ \say{down} (in direction $-e$) into the original domain while leaving the \say{lower} face of $Q$ fixed. 
\begin{figure}\label{fig:domainshrinking}
\centering
\begin{tikzpicture}[scale=1.2,
declare function = {ssmooth(\z) = 2*\z*\z*(1-1/2*\z*\z/9)/9;},
declare function = {freq(\z) = (\z+4)/3*\z;},
declare function = {bvy(\x,\z) = 4+(1-\z)*(\oscila*sin(\freqa*freq(\x)+\offsa)+\oscilb*sin(\freqb*(\x)))
	+\z*(\oscila*sin(\freqa*freq(\x+1)+\offsa)+\oscilb*sin(\freqb*(\x+1)));}
]
\def\oscila{0.8}
\def\oscilb{1.2}
\def\freqa{50}
\def\freqb{19}
\def\offsa{40}
\def\shrink{4/5}

\draw [fill=white!60!white] (-4,6) -- (4,6) -- (4,0) -- (-4,0) -- cycle;
\draw [->,thick] (-2.5,4.7) --++ (0,0.3)  node [anchor = east] {$e$} --++ (0,0.3);
\node [anchor = north west] at (-4,6) (psi) {$Q$};
\foreach \x in {-3,...,-1}
    \fill [fill=gray!40!white] ({\x-1},0) -- ({\x-1},{bvy(\x-1,0)}) -- ({\x},{bvy(\x,0)}) -- ({\x},0) -- cycle;

\def\ssize{0.25}
\def\xsize{1}
\newcommand{\smacro}{\ssize, 2\ssize,\xsize}

\foreach \x in {-1,...,1}
    \foreach \s in {0.25,0.5,...,1}
        \fill [fill=gray!40!white] ({\x+\s-\ssize},0) 
        -- ({\x+\s-\ssize},{(1-(1-\shrink)*ssmooth(\x+\s-\ssize+1))*bvy(\x,\s-\ssize)}) 
        -- ({\x+\s},{(1-(1-\shrink)*ssmooth(\x+\s+1))*bvy(\x,\s)}) 
							-- ({\x+\s},0) -- cycle;

\foreach \x in {-1,...,1}
	\foreach \s in {\ssize,0.5,...,\xsize}
		\draw [color=black,thick]
            ({\x+\s-\ssize},{(1-(1-\shrink)*ssmooth(\x+\s-\ssize+\xsize))*bvy(\x,\s-\ssize)}) -- ({\x+\s},{(1-(1-\shrink)*ssmooth(\x+\s+\xsize))*bvy(\x,\s)});

\def\ssize{0.25}            
			\foreach \x in {3,...,4}
					\fill [fill=gray!40!white] ({\x-1},0) -- ({\x-1},{\shrink*bvy(\x-1,0)}) -- ({\x},{\shrink*bvy(\x,0)}) -- ({\x},0) -- cycle;
			\foreach \x in {3,...,4}
					\draw [color=black,thick] ({\x-1},{\shrink*bvy(\x-1,0)}) -- ({\x},{\shrink*bvy(\x,0)});

			\foreach \x in {-3,...,4}
		    	\draw[color=gray,thick] ({\x-1},{bvy(\x-1,0)}) -- ({\x},{bvy(\x,0)});
			\node at ({3},{bvy(3,0)+0.2}) (boundary) {$\partial\Omega$};
			\foreach \x in {-3,...,-1}
		    	\draw[color=blue,thick] ({\x-1},{bvy(\x-1,0)}) -- ({\x},{bvy(\x,0)});
			\node at ({-3},{bvy(-3,0)+0.4}) (bgamma) {{\color{blue}$\Gamma$}};
			
			\foreach \x in {-4,-3,-2}
				\foreach \s in {0.5,1}
					\draw[color=blue,dashed,thin] ({\x+\s},0) -- ({\x+\s},{bvy(\x,\s)});
			\foreach \x in {-1,0  ,...,3}
				\foreach \s in {0.5,1}
					\draw[color=gray,dashed,thin] ({\x+\s},0) -- ({\x+\s},{bvy(\x,\s)});
					
			\node at (1,2) (psi) {$\hat\Psi_j(\Omega;Q)$};
\end{tikzpicture}
\caption{Locally shrinking $\Omega$ with factor $\frac{j-1}{j}=\frac{4}{5}$.}
\end{figure}
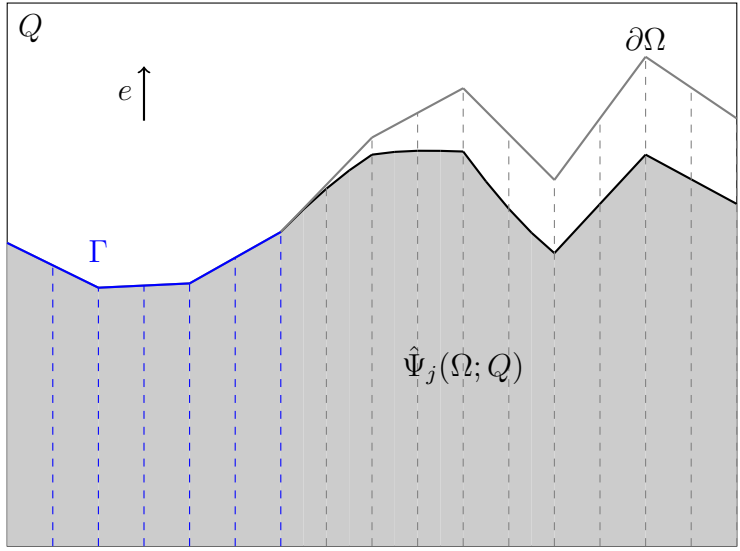

Clearly, $\hat{\Psi}_j(\cdot;Q)$ is of class $C^\infty$ on $Q$. Moreover, we can choose 
$i(j)\to \infty$ as $j\to\infty$, slowly enough so that still
$\norm{\nabla^m \eta_{i(j)}}_{L^\infty} \leq C(m) j^{\frac{1}{2}}$, for each $j,m\in \NN$ with constants $C(m)>0$ independent of $j$.
Thus
$\absb{1-\frac{j-1}{j}} \norm{\nabla^m \eta_{i(j)}}_{L^\infty}\leq C(m) j^{-\frac{1}{2}} \to 0$ as $j\to \infty$, for each $m$, and as a consequence,
we obtain that $\hat{\Psi}_j(\cdot;Q)\to \identity|_Q$ in $C^m$.
Since $\partial\Omega$ can be covered by finitely many such cubes, we can write $\overline{\Omega}\subset Q_0\cup \bigcup_{k=1}^{n} Q_k$ with some open interior set $Q_0\subset\subset \Omega$. For a smooth decomposition of unity $1=\sum_{k=0}^n \varphi_k$ subordinate to this covering of $\Omega$ (i.e., $\varphi_k$ smooth, non-negative and compactly supported in $Q_k$), 
\[
    \Psi_j(x):=\varphi_0(x)x+\sum_{k=1}^n \varphi_k(x) \hat{\Psi}_j(x;Q_k)
\]
now has the asserted properties, with $\Gamma_j:=\Gamma_{i(j)}$.
\end{proof}
\begin{proof}[Proof of Lemma~\ref{lem:shrinkcont}]
We only provide a proof for the case $k=1$, which will include the argument for $k=0$. For $k\geq 2$, the assertion follows inductively.
It suffices to show that as $j\to\infty$, $\partial_n [f\circ \Psi_j-f] \to 0$ in $L^p$, 
for each partial derivative $\partial_n$, $n=1,\ldots, d$. By the chain rule, 
\begin{align}\label{lemshrinkdef-1}
\bald
	&\partial_n [f\circ \Psi_j-f]=
	[(\nabla f)\circ \Psi_j] \partial_n \Psi_j -\partial_n f \\
	&\quad =\Big([(\nabla f)\circ \Psi_j] \partial_n \Psi_j-(\partial_n f)\circ \Psi_j \Big)
	+\Big((\partial_n f)\circ \Psi_j -\partial_n f\Big)
\eald
\end{align}
The first term above converges to zero in $L^p$ since 
$(\nabla f) e_n=\partial_n f$ for the $n$-th unit vector $e_n$, and
$\partial_n\Psi_j \to \partial_n\identity=e_n$ uniformly. The convergence of the second term correspond to our assertion for the case $k=0$, with 
$\tilde{f}:=\partial_n f\in L^p$. It can be proved in the same way as the well-known continuity of the shift in $L^p$:
If $\tilde{f}$ is smooth and can be extended to a smooth function on $\RR^d$, we have
\begin{align}\label{lemshrinkdef-2}
	\normn{\tilde{f}\circ \Psi_j -\tilde{f}}_{L^p(\Omega;\RR^d)}\leq \normn{\nabla\tilde{f}}_{L^\infty(\RR^d;\RR^{d\times d})} \norm{\Psi_j-\identity}_{L^p(\Omega;\RR^d)}\underset{j\to\infty}{\To} 0.
\end{align}
The general case follows by approximation of $\tilde{f}$ in $L^p$ with such smooth functions, by first extending 
$\tilde{f}$ by zero to all of $\RR^d$, and then mollifying.
Here, notice that for the mollified function, $\normn{\nabla\tilde{f}}_{L^\infty}$ in \eqref{lemshrinkdef-2} is unbounded in general as a function of the mollification parameter, but one can always choose the latter to converge slow enough with respect to $j$ so that \eqref{lemshrinkdef-2} still holds.
\end{proof}

\subsection{Linear elasticity with a constraint guaranteeing local invertibility}\label{ssec:LE}

Linear elasticity, a very popular model due to its computational efficiency, constitutes a standard quadratic approximation 
of nonlinear elastic energies for the case where $\nabla y$ is close to the identity matrix $I$.
Still, this potentially allows $y$ to deviate from the identity quite a lot, depending on the size of the reference domain $\Omega$, so that a global self-contact problem remains meaningful. In this model, we work with an elastic part of the energy of the form
\newcommand{\Ele}{E^{\rm el}_{\rm lin}}
\begin{equation}\label{energy_elastic}
	\Ele(y)=\int_\Omega Q(\nabla y)\dx.
\end{equation}
The  prototypical example for the  density $Q$ is a quadratic function
\begin{equation}\label{density}
    Q(\nabla y)=  \mu \abs{e}^2 + (\lambda/2) \tr^2(e)   
\end{equation}
of the symmetrized small strain tensor
\begin{equation*}
e=\frac{1}{2}[(\nabla y - I)^\top+(\nabla y - I)]
\end{equation*}
with the Lam\'e material parameters
\begin{equation}\label{Lame}
    \mu> 0 \quad\text{and}\quad \lambda>-\frac{2}{d}\mu.
\end{equation}
In particular, $Q$  as given in \eqref{density}  is strictly convex in $e$ for precisely the range of $\lambda$, $\mu$ stated above, and the existence and uniqueness of solutions (minimizers or critical points) is guaranteed for suitable boundary conditions and added linear potentials  corresponding to external forces, essentially using Korn's inequality 
 \cite{Ni81a}. While $Q$ as defined above corresponds to an isotropic material, we could just as well allow anisotropic cases.  For instance,  $Q$ can be  also be  chosen as any symmetric quadratic form which is a strictly convex function of $e$.  In fact, we do not even need that $Q$ is quadratic, \eqref{Q} below suffices. 

This model is well understood, see, e.g., \cite[Sec.~6.3]{Cia88B} for the case $d=3$, and asymptotically justified as a limit of nonlinear elasticity \cite{DalNePe02},  provided that Dirichlet conditions are imposed on a part of the boundary.
Pure traction problems are also possible but more subtle, see Remark~\ref{rem:puretraction}.  
Note however that such asymptotic results does not automatically provide insight on self-contact problems. In fact, it is not hard to see that a constraint like the Ciarlet-Ne\v{c}as condition would simply vanish in the passage to a linearized limit model unless the identity (or another rotation that we linearize at) is already in self-contact. Concerning the latter, the only related result so far available which rigorously justifies a linearized contact condition is given in \cite{AlDaFrie22Pa} and essentially limited to dimension $2$. 
As we intend to look at scenarios where the deformation is not really asymptotically small, but just has deformation gradients  moderately  close to the identity, we still aim to reproduce global injectivity as a constraint, without trying to linearize it.

An obvious issue of the model is that by itself, the linear elastic energy $\Ele$ does not enforce invertibility of the deformation $y$, neither locally nor globally. We will
therefore augment it with additional constraints 
that prevent interpenetration of matter at least locally, to make sure that the global self-contact problem remains meaningful and can be tackled with penalization. 
More precisely, we 
will admit only locally orientation preserving deformations (i.e., $\det\nabla y>0$ a.e.~in $\Omega$) also satisfying the  
constraints
\begin{align}\label{idangleconstraint}
\bald
    &l \abs{x_2-x_1}^2 \leq 
		(y(x_2)-y(x_1))\cdot (x_2-x_1),\\
		&\abs{y(x_2)-y(x_1)}\leq L	\abs{x_2-x_1},\\
	&\text{for a.e.~$x_1,x_2\in \Omega$ with $\abs{x_1-x_2}<\varrho$},
\eald
\end{align}
where $\varrho,l,L>0$ are given, fixed constants.
In particular, \eqref{idangleconstraint} implies that $y$ is uniformly locally bi-Lipschitz:
\begin{align}\label{biLi}
\bald
    l \abs{x_1-x_2} \leq 
	\abs{y(x_1)-y(x_2)} \leq L \abs{x_1-x_2}&\\
	\text{for a.e.~$x_1,x_2\in \Omega$ with $|x_1-x_2|<\varrho$}&.
\eald
\end{align}
Consequently, $|\nabla y|\leq L$ and $|(\nabla y)^{-1}|\leq l^{-1}$ in operator norm,
and for locally orientation preserving deformations $y$, $\det\nabla y\geq l^d>0$. 
Our constraints thus contain the local invertibility constraint studied in \cite{FoRo01a} as well as the locking constraints of \cite{BeKruSchloe18a}.
Such invertibility properties of $\nabla y$ by themselves would not suffice for our purposes, however, because they do not provide enough control near the boundary. By contrast, \eqref{idangleconstraint} also prevents local self-intersections near the boundary that one could otherwise create by closing outer angles. In a nonlinear elastic settings, additional constraints like \eqref{idangleconstraint} 
can possibly be avoided as shown in \cite{KroeRei22Pa}, based on the theory of functions of finite distortion \cite{HeKo14B}. However, it is doubtful whether the approach of \cite{KroeRei22Pa} would work with our penalty term here.

While further restrictions on $l$ and $L$ are not required for our theory (as long as $L>l$ so that some admissible $y$ exist), for the model it is reasonable to have $l<1<L$ and close, while $\varrho$ can be very small. With this, \eqref{idangleconstraint} implies that $\nabla y$ cannot be too far from the identity, 
with the distance controlled by $l$ and $L$. From this point of view, one can also consider \eqref{idangleconstraint} as a way to specify the range for which we consider the linear elastic model an acceptable approximation of the nonlinear case.
Besides local invertibility, \eqref{idangleconstraint} also provides additional regularity of deformations so that
the pointwise definition of $\det \nabla y$ is meaningful and the Ciarlet-Ne\v{c}as condition
\eqref{ciarletnecas} implies global invertibility a.e.. 
The natural space $W^{1,2}$ would be too weak for these purposes in dimension $d\geq 3$.  We refer to 
 \cite{GiAlPo08a} for some results with minimal requirements.

The full internal energy with an added nonlocal surface term penalizing self-interpenetration on the boundary now reads
\begin{align*}
	E_{\epsp}(y)=\left\{\begin{alignedat}[c]{2}
	&\Ele(y)+E^{\partial\Omega}_{\epsp}(y)\quad && \text{if $y\in \cY$},\\
		&+\infty && \text{else,}
\end{alignedat}\right.	
\end{align*} 
where for some fixed parameters $\varrho,l,L>0$,
\[
   \cY=\cY(\varrho;l,L):=\mysetr{y \in W^{1,2}(\Omega;\RR^d)}{\text{\eqref{idangleconstraint} holds}
	}
\]
Notice that 
\[
	\cY\subset \BiLip(\Omega;\varrho;l,L):=\{y\in W^{1,2}(\Omega;\RR^d)\mid \text{\eqref{biLi} holds}\},
\]
 $\cY\subset W^{1,\infty}$  and for all $y\in \cY$, we automatically have that $\det\nabla y\geq l^d$. 

Below, we will also need the subset of $\cY$ where the constraint \eqref{idangleconstraint} is inactive in the sense that
it holds with slightly stronger constants, i.e,
\[
   \hat{\cY}=\hat{\cY}(\varrho;l,L):=\bigcup_{t>1}\cY_t,\quad\text{where}~~\cY_t:=\cY(\varrho;l t,L/t)\subset \cY(\varrho;l,L).
\]


We claim that as $\epsp\to 0$,
$E_{\epsp}$ converges in a suitable sense to
\begin{align} 
	E(y):=\left\{\begin{alignedat}[c]{2}
	  &\Ele(y)\quad && \text{if $y\in \cY$ and \eqref{ciarletnecas} holds},\\
		&+\infty && \text{else,}
	\end{alignedat}\right.
\end{align} 
the original linear elastic energy with the Ciarlet-Ne\v{c}as condition \eqref{ciarletnecas} added as a constraint.

\begin{rem}\label{rem:constraintclosed}
The set of admissible deformations $\cY$ is closed with respect to weak convergence in $W^{1,2}$,
see the proof of Theorem~\ref{thm:convergenceLE} (i) below. It is therefore compatible with direct methods for the existence of minimizers.
 In particular, $E$ always has a minimizer obtained as a $W^{1,2}$-weak limit of a minimizing sequence. Here, notice that 
since $\cY$ is a bounded subset of $W^{1,\infty}$ (possibly up to rigid translations that can be removed,
cf.~Remark~\ref{ex:min-pen}), neither additional coercivity properties nor Korn's inequality are needed. Moreover, $W^{1,2}$-weak convergence of a sequence in $\cY$ 
automatically implies $W^{1,p}$-weak convergence for all $1\leq p<\infty$. Using $p>d$, we see that the Ciarler-Ne\v{c}as conditon \eqref{ciarletnecas} is preserved in the limit, too \cite{CiaNe87a}. 
\end{rem}

\begin{rem}\label{rem:biLiLavretiev}
From a purely theoretical perspective,  we could also replace $\cY$ with the larger set of orientation preserving maps satisfying the weaker condition \eqref{biLi},  the only property we really exploit. However, \eqref{idangleconstraint} is better for numerical results because as we will see, it admits (approximately) conforming finite elements. This seems to be unclear for \eqref{biLi}, where, as far as we known, it is unknown whether or not a Lavrentiev phenomenon could occur in our context, at least in dimension $d\geq 3$. For $d=2$, see \cite{DaPra14a} for a related approximation result.
\end{rem}

As before, discrete Galerkin-type approximations can be included by further restricting $E_{\epsp}$. 
Let $h>0$ (typically a mesh size) and let $Y_h$ be an associated finite dimensional subspace of 
$W^{1,2}(\Omega;\RR^d)$ such that 
the approximation error
 satisfies
\begin{align}\label{fespaceapproxLE}
\begin{aligned}	\cE_h(y):=
	\sup_{\tau>1}\Big(\inf_{y_h\in Y_h\cap \cY(\varrho;\tilde{l}/\tau,\tilde{L}\tau)} \norm{y-y_h}_{W^{1,2}} 
	\Big)\underset{h\to 0}{\To} 0&\\
	\text{for all}~~\tilde{l},\tilde{L}~~\text{with}~~0<l<\tilde{l}<\tilde{L}\leq L~~\text{and all}~~y\in \tilde \cY(\varrho;\tilde{l},\tilde{L})&.
\end{aligned}
\end{align}
In fact, it is enough to have the the above for $(\tilde{l},\tilde{L})$ in a small neighborhood of $(l,L)$.

The corresponding finite dimensional approximation of $E_{\epsp}$ is defined by 
\begin{align*}
	E^h_{\epsp}(y):=\left\{\begin{alignedat}[c]{2}
	&\Ele(y)+E^{\partial\Omega}_{\epsp}(y)\quad && \text{if $y\in Y_h\cap \cY$},\\
		&+\infty && \text{else,}
\end{alignedat}\right.	
\end{align*} 

 In view of the rather strong restriction \eqref{biLi},
we can actually work with a much more general class of densities $Q$, not necessarily quadratic: 
\begin{align}
	\begin{aligned} &Q:\RR^{d\times d}\to \RR~~~~~\text{is continuous and quasiconvex.} \\
	\end{aligned} \label{Q}
\end{align}
As a consequence of \eqref{Q}, $y\mapsto \Ele(y)=\int_\Omega Q(x,\nabla y)\dx$ is
$W^{1,\infty}$-weak$^*$ sequentially lower semicontinuous (see \cite[Theorem 8.4]{Da08B}, e.g.),
and also $W^{1,p}$-strongly continuous on $\cY$ (or any other bounded subset of $W^{1,\infty}$) for arbitrary $1\leq p<\infty$.
Concerning $\Ele$, this is all we will need below.

Our main result is the following.

\begin{thm}\label{thm:convergenceLE}
 Let  $\Omega\subset \RR^d$ be a bounded Lipschitz domain such that 
$\RR^d\setminus \partial\Omega$ has only two connected components, and let $\Ele$ (see \eqref{energy_elastic},  \eqref{Q}), $E^{\partial\Omega}_{\epsp}$ (see \eqref{sCN2}, \eqref{sCN1b}, \eqref{sCN1c}), $\cY$, $\hat{\cY}$, $E$ and $E^h_{\epsp}$ be given as above. In addition, assume that \eqref{fespaceapproxLE} holds  and that 
the constant $\beta$ in $E^{\partial\Omega}_{\epsp}$ satisfies $\beta>d-1$. 
For every sequence $(h(k),\epsp(k))\in (0,\infty)^2$, 
$k\in\NN$, with 
$h(k)\to 0$ and $\epsp(k)\to 0$ as $k\to \infty$, we then have the following properties for all
$y\in W^{1,2}(\Omega;\RR^d)$:
\begin{enumerate}
\item[(i)] For every sequence $y_k\rightharpoonup y$ in $W^{1,2}$ (weakly),
\[
	\liminf_{k\to\infty} E_{\eps(k)}^{h(k)}(y_k)\geq E(y);
\]
\item[(ii)$_1$] if $y\in \hat{\cY}$ or $y$ does not satisfy \eqref{ciarletnecas}, then
there exists a sequence $y_k\to y$ in $W^{1,2}$ (strongly) such that
\[
	\lim_{k\to\infty} E_{\eps(k)}^{h(k)}(y_k)= E(y);
\]
\item[(ii)$_2$] if $y\in \cY\setminus \hat{\cY}$ and $y$ satisfies \eqref{ciarletnecas}, then for every $\tau>1$
there exists a sequence 
$y_k\in Y_{h(k)}$ with $y_k\to y$ in $W^{1,2}$ (strongly),
\[
	y_k\in \cY(\varrho,l/\tau,L\tau)~~\text{and}~~ \Ele(y_k)+E^{\partial\Omega}_{\epsp}(y_k)\underset{k\to\infty}{\To} \Ele(y).
\]
\end{enumerate}
This also remains true for the case $h=0$ 
if we define $E^0_{\epsp}:=E_{\epsp}$, and the assumption \eqref{fespaceapproxLE} can be dropped in this case. 
\end{thm}
\begin{rem}\label{rem:constrainedFE}
Finite elements fully conforming with $\cY$, i.e., \eqref{fespaceapproxLE} with fixed $\tau=1$ instead of relaxing the constraint constants using $\tau>1$, seem to be hard to obtain if not impossible.
In Proposition~\ref{prop:confFE} at the end of this section, \eqref{fespaceapproxLE} is proved for piecewise affine elements on simplicial meshes. 
\end{rem}
\begin{rem}\label{rem:BiLiconstraintactive}
Theorem~\ref{thm:convergenceLE} (ii)$_2$ is weaker than one might like, because our construction only yields the constraint \eqref{idangleconstraint}
for $y_k$ in slightly relaxed form, with constants modified by a factor $\tau>1$. As a consequence, it is too weak to ensure finite $E_{\epsp}^{h(k)}(y_k)$. On the technical level, we do not know how to handle the scenario where \eqref{ciarletnecas} holds with self-contact on the surface while the deformation maximally stretches and compresses locally nearby, say, one end forming a slim-necked peg which, after deformation, is maximally compressed and stuck in a superficial, maximally stretched opening somewhere else on the body. In such an example, it is possible that $E_{\epsp}$ is infinite in a whole neighborhood of $y$ for all $\epsp>0$, and the correct value of the limit functional at such a $y$ should read $E(y)=+\infty$, not $E(y)=\Ele(y)$ as we defined $E$.
\end{rem}

\begin{proof}[Proof of Theorem~\ref{thm:convergenceLE}]
{\bf (i) \say{Lower bound}:} Let $y_k\rightharpoonup y$ be $k\to\infty$, weakly in $W^{1,2}$. 
Passing to a subsequence if necessary, we may assume that 
$e_0:=\liminf E^{h(k)}_{\epsp(k)}(y_k)$ is a limit. Moreover, 
we may assume that $e_0<\infty$, because
otherwise the asserted lower bound is trivial.
As a consequence, $y_k\in \cY$ for all sufficiently large $k$, because $E^{h(k)}_{\epsp(k)}=+\infty$ on $W^{1,2}\setminus \cY$. 
 In particular, $(y_k)$ is bounded in $W^{1,\infty}$ and we also have that $y_k\rightharpoonup^* y$ in $W^{1,\infty}$. 
By compact embedding, again passing to a subsequence if necessary,
we have $y_k\to y$ uniformly on $\overline{\Omega}$. In particular, the limit $y$ also 
satisfies \eqref{idangleconstraint},  whence  $y\in \cY$.

With $K:=e_0+1$, we also have that $E^{h(k)}_{\sigma,\eps(k)}(y_k)\leq K$ for all $k$ sufficiently large.
As a consequence of this energy bound and the fact that $\Ele$ is bounded from below  on $\cY$,  
$E^{\partial\Omega}_{\epsp(k)}(y_k)$ is bounded, so that $y_k$ is invertible on $\partial\Omega$ by Corollary~\ref{cor:boundaryinvert} for all large enough $k$. Hence, $y\in \AIB$ (cf.~Definition~\ref{def:AIB} and Remark~\ref{rem:AIB-CN}) and $y$ satisfies the Ciarlet-Ne\v{c}as condition \eqref{ciarletnecas} due to Theorem~\ref{thm:AIB-CN}. As $\Ele$ is  weakly$^*$ sequentially lower semicontinuous in $W^{1,\infty}$,  we infer that
\[
	\liminf E^{h(k)}_{\epsp(k)}(y_k)\geq \liminf \Ele(y_k)\geq \Ele(y)=E(y).
\]

{\bf (ii)$_1$ recovery sequence for inactive constraint:}
If \eqref{ciarletnecas} does not hold, we may choose $y_k\equiv y$. Otherwise, $y:\Omega\to y(\Omega)$ is a homeomorphism, as a locally bi-Lipschitz map satisfying \eqref{ciarletnecas}. 
 To obtain an approximation of $y$ with a controllable contribution in $E^{\partial\Omega}_{\epsp(k)}$, we   
first  create a small gap all around the boundary, 
using the smooth injective maps $\Psi_j:\bar\Omega\to \Omega$ close to the identity from Lemma~\ref{lem:shrink} which shrink the Lipschitz domain $\Omega$ into itself. $y\circ \Psi_j$ fully avoids self-contact: $y\circ \Psi_j(\overline{\Omega})$ is compactly contained in the open set $y(\Omega)$ and $y\circ \Psi_j$ is still locally bi-Lipschitz. In particular, $E_{\epsp}^{\partial\Omega}(y\circ \Psi_j)=0$ for all $\epsp>0$ small enough by Corollary~\ref{cor:whenECNvanishes}. Moreover, 
$y\in \cY(\varrho;lt,L/t)$ for some $t>1$, and thus
$y\circ \Psi_j \in \hat\cY$ for all large enough $j$ because 
$\Psi_j\to\identity$ in $C^1$. In addition, $y\circ \Psi_j\to y$ in $W^{1,2}$ by Lemma~\ref{lem:shrinkcont}.
For each $j$, we approximate $y\circ \Psi_j$ with suitable finite elements $y_{j,k}\in Y_{h(k)}\cap \hat\cY$ according to \eqref{fespaceapproxLE}.
This leads to a sequence $y_{j,k}\in Y_{h(k)}\cap \cY$ such that as $k\to\infty$, $y_{j,k}\to y\circ \Psi_j$ in $W^{1,2}$, 
and for all $k$ large enough (depending on $j$) $E_{\epsp(k)}^{\partial\Omega}(y_{j,k})=0$.
Moreover, $\lim_j \lim_k \Ele(y_{j,k})=\lim_j \Ele(y\circ \Psi_j)=\Ele(y)$ by continuity of $\Ele$  on $\cY$ with respect to the strong topology of  $W^{1,2}$. 
With a suitable diagonal sequence $y_k:=y_{j(k),k}$ with $j(k)\to\infty$ (slow enough),
$y_k\to y$ in $W^{1,2}$ and $E_{\epsp(k)}^{h(k)}(y_k)=\Ele(y_k)\to \Ele(y)$ as asserted.

{\bf (ii)$_2$ recovery sequence with weaker constraint constants:} 
This is fully analogous to $(ii)_1$. We now use \eqref{fespaceapproxLE} with $\tilde{l}=l$ and
$\tilde{L}=L$.
\end{proof}

\subsection{Nonlinear elasticity with higher order terms}\label{ssec:NLEnonsimple}
As an alternative model, we briefly revisit the scenario already studied in \cite{KroeVa19a}. 
There, artificial constraints like \eqref{idangleconstraint} or \eqref{biLi} are not imposed. Instead,
a higher order term is added to a nonlinear elastic energy, leading to a model of a so-called non-simple material. 
The local bi-Lipschitz property required for our analysis of the boundary penalty term in Theorem~\ref{thm:cnsoftboundary}
will now be obtained as consequence of an energy bound.

For $y\in W^{1,p}(\Omega;\RR^d)$, consider
the penalized energy given by
\begin{align*}
	 E_{\epsp,\sigma}(y)=\left\{\begin{alignedat}[c]{2}
	&  E^{el}(y) 	+E^{reg}_{\sigma}(y)+E^{\partial\Omega}_{\epsp}(y)\quad && \text{if $y\in W^{2,s}$},\\
		&+\infty && \text{else,}
\end{alignedat}\right.	
\end{align*} 
in the limit as  $\epsp\to 0$. 
We will see that it converges to
\begin{align*}
	E_{\sigma}(y)=\left\{\begin{alignedat}[c]{2}
	  &E^{el}(y) +E^{reg}_{\sigma}(y)\quad && \text{if $y\in W^{2,s}$ and \eqref{ciarletnecas} holds},\\
		&+\infty && \text{else,}
	\end{alignedat}\right.
\end{align*} 
the original energy which includes the Ciarlet-Ne\v{c}as condition \eqref{ciarletnecas} as a built-in constraint.
 Here, 
\[
	E^{el}(y)=\int_\Omega W( \nabla y)\dx.
\]
 
where 
\begin{align}\label{W0}
\begin{aligned}
	&  W: \RR^{d\times d} \to \RR\cup \{+\infty\}~~\text{is continuous} \\
\end{aligned}
\end{align}
Moreover, for all  $F\in \RR^{d\times d}$,  
\begin{align}\label{W1}
&\begin{alignedat}{2}
	+\infty\,=\,&  W(F)  && \quad\text{if $\det F\leq 0$,}\\
	+\infty\,>\,&  W(F)  \,\geq\, c_1\left(\abs{F}^p+(\det F)^{-q}\right)-c_2 && \quad\text{if $\det F>0$,}\\
\end{alignedat}
\end{align}
with constants $q>d$ (which is necessary for \eqref{pqsd} below), $c_1>0$ and $c_2\geq 0$. 
In addition, we assume that $W$ is polyconvex, i.e., 
\begin{align}\label{W2}
\bald
   W(F)=\Psi(m(F)),~~&\text{with a  convex function $\Psi$,}
\eald
\end{align}
where $m(F)\in \RR^{n(d)}$, $n(d):=\sum_{k=1}^{d} \binom{d}{k}^2$, denotes the collection of all minors of $F$, i.e., all $k\times k$ sub-determinants with $1\leq k\leq d$. 
This means that for $d=2$, $m(F)=(F,\det F)\in \RR^{5}$ and for $d=3$, $m(F)=(F,\cof F, \det F)\in \RR^{19}$. Here,
$\cof F\in \RR^{d\times d}$ denotes cofactor matrix so that
$F^{-1}=(\cof F)^T (\det F)^{-1}$ whenever $F$ is invertible. 

 It would also be possible to use a more general $W$ with explicit dependence on $x$ or approximate it with truncated, everywhere finite integrands $W_{\epse}$ (which are safer for numerical purposes), subsequently considering the simultaneous limit $(\epse,\epsp)\to (0,0)$
as in \cite{KroeVa19a}. As our numerical experiment 
are not conducted in this framework, however, we will not further discuss these generalizations here. 

\begin{rem}
Due to \cite{Ba77a,CiaNe87a}, $E^{el}$ always has a minimizer $y^*$ in $W^{1,p}(\Omega;\RR^d)$ even if partial Dirichlet boundary conditions or compact force terms are added. Like all states with finite energy, it must satisfy 
$\det \nabla y^*>0$ a.e.~in $\Omega$.
\end{rem}
With the penalty proposed here, classical nonlinear elasticity seems to be out of reach for our analysis concerning the approximation of invertibility with surface penalization without additional constraints.  We thus further
modify the elastic energy by adding a regularizing term. 
 Accordingly, we regularize $E^{el}$ by adding 
the  higher order term
\begin{align*}
	E^{reg}_{\sigma}(y):=\sigma\int_\Omega \abs{D^2 y}^s \dx,
\end{align*}
 with a fixed parameter $\sigma>0$.
Altogether, the exponents are assumed to satisfy
\begin{align}\label{pqsd}
	p>d, \quad s>d, \quad q>\frac{sd}{s-d}.
\end{align}
Here, as in \cite{KroeVa19a}, the latter ensures that
 $q,s$ are admissible for 
 the result of
\cite{HeaKroe09a} 
summarized in Lemma~\ref{lem:HK},  so that we can obtain a uniform lower bound for $\det \nabla y$. 

Discrete Galerkin-type approximations can also be included. 
For that, let
$h>0$ (typically a mesh size) and let $Y_h$ denote an associated finite dimensional subspace of 
$(W^{2,s}\cap W^{1,p})(\Omega;\RR^d)$ (typically $Y_h\subset W^{2,\infty}$) such that
for each $y$, the approximation error $\cE(y;h)$ satisfies
\begin{align}\label{fespaceapprox}
	\cE(y;h):= \inf_{y_h\in Y_h}\big(\{\norm{y-y_h}_{W^{2,s}\cap W^{1,p}}\big)
	\underset{h\to 0}{\To} 0.
\end{align}
The corresponding finite dimensional approximations of $E_{\epsp,\sigma}$ are
\begin{align*}
	E^h_{\epsp,\sigma}(y):=\left\{\begin{alignedat}[c]{2}
	&E^{el}(y)+E^{reg}_{\sigma}(y)+E^{\partial\Omega}_{\epsp}(y)\quad && \text{if $y\in Y_h$},\\
		&+\infty && \text{else,}
\end{alignedat}\right.	
\end{align*} 
For this model, we have convergence of  $E^h_{\epsp,\sigma}$  to  $E_{\sigma}$ in the following sense:
\begin{thm}\label{thm:convergence}
Let $\sigma>0$ and $\beta>d-1$ be fixed, let $\Omega\subset \RR^d$ be a bounded Lipschitz domain such that 
$\RR^d\setminus \partial\Omega$ has only two connected components,
and assume that \eqref{W0}--\eqref{pqsd} hold.
For every sequence  $(h(k),\epsp(k))\in (0,\infty)^2$, 
$k\in\NN$, with 
$h(k)\to 0$ and $\epsp(k)\to 0$ as $k\to \infty$, we have the following two properties for all
$y\in W^{2,s}(\Omega;\RR^d)$:
\begin{enumerate}
\item[(i)] For every sequence $y_k\rightharpoonup y$ in $W^{2,s}$ (weakly),
\[
	\liminf_{k\to\infty} E_{\epsp(k),\sigma}^{h(k)}(y_k)\geq E_{\sigma}(y);
\]
\item[(ii)] there exists a sequence $y_k\to y$ in $W^{2,s}$ (strongly) such that
\[
	\lim_{k\to\infty} E_{\epsp(k),\sigma}^{h(k)}(y_k)= E_{\sigma}(y).
\]
\end{enumerate}
This also holds for the  case $h=0$, 
if we define $E^0_{\epsp,\sigma}:=E_{\epsp,\sigma}$. 
\end{thm}

\begin{proof}
We provide only a proof for the  case 
including  Galerkin approximations with $h(k)>0$, $h(k)\to 0$. 
The case $h=0$ is similar and even slightly simpler.
{\bf (i) \say{Lower bound}:} Let $y_k\rightharpoonup y$ as $k\to\infty$, weakly in $W^{2,s}$. By compact embedding, this implies that
$y_k\to y$ strongly in $W^{1,\infty}$.
Passing to a suitable subsequence (not relabeled), we may assume that
$e_0:=\liminf E^{h(k)}_{\sigma,\eps(k)}(y_k)=\lim E^{h(n)}_{\sigma,\eps(n)}(y_k)$. In addition, we may assume that 
$e_0<+\infty$ because otherwise there is nothing to show. With $K:=e_0+1$,
we have $E^{h(k)}_{\sigma,\eps(k)}(y_k)\leq K$ for all $k$ sufficiently large.

As a consequence of the energy bound, \eqref{yconstants} holds for $y=y_k$. More precisely, the bounds on $\nabla y_k$
follow by embedding and the fact that $E_\sigma^{reg}(y_k)$ and $E^{el}(y_k)$ control $\norm{D^2 y_k}_{L^s}$ and $\norm{\nabla y_k}_{L^p}$, respectively.
The lower bound on $J:=\det\nabla y_k$ is provided by 
 Lemma~\ref{lem:HK} 
combined with our assumptions \eqref{W1} and \eqref{pqsd} on $W$. Given $\delta>0$ and $M_1$ as in \eqref{yconstants}, \eqref{W0} and \eqref{W1} imply that 
\begin{align}\label{Wlocunicont}
	\text{$W$ is uniformly continuous on}~~\mysetr{F\in \RR^{d\times d}}{\begin{array}[c]{l} \abs{F}\leq M_1,\\ \det F\geq \delta\end{array}}.
\end{align}
As $\nabla y_k\to \nabla y$ in $L^\infty$, we infer that
\begin{align}\label{tconv-1}
	E^{el}(y_k)=\int_\Omega W(\nabla y_k)\dx \underset{k\to\infty}{\To} \int_\Omega W(\nabla y)\dx=E^{el}(y).
\end{align}
Moreover, by the weak lower semicontinuity of the convex functional $E^{reg}_\sigma$,
\begin{align}\label{tconv-2}
	\liminf_{k\to\infty} E^{reg}_\sigma(y_k)\geq E^{reg}_\sigma(y).
\end{align}
In addition to \eqref{tconv-1} and \eqref{tconv-2}, 
it also trivially holds that $E^{\partial\Omega}_{\epsp(k)}\geq 0$. Thus, we conclude that
$\liminf E^{h(k)}_{\sigma,\eps(k)}(y_k)\geq E_\sigma(y)$ as asserted, provided that
$y$ satisfies the Ciarlet-Ne\v{c}as condition. The latter follows from Theorem~\ref{thm:AIB-CN}, Remark~\ref{rem:AIB-CN} and Corollary~\ref{cor:boundaryinvert}.

{\bf (ii) Existence of a strongly converging recovery sequence:}
 We  may assume that $E_\sigma(y)<+\infty$, because otherwise 
$E_{\eps(k),\sigma}^{h(k)}(y)\to+\infty=E_\sigma(y)$.
Hence, $y$ satisfies \eqref{ciarletnecas}.
As $s>d$, $y$ is also $C^1$.  Moreover, 
the fact that $E^{reg}_\sigma(y)+E^{el}(y)<+\infty$ implies that 
$\inf \det\nabla y>0$ by 
 Lemma~\ref{lem:HK} ($s$, $\alpha:=\frac{s-d}{s}$  and $q$ are admissible for this result due to \eqref{pqsd}).
We infer that $y(\Omega)$ is open and $y:\Omega\to y(\Omega)$ is a homeomorphism that is locally uniformly bi-Lipschitz by
Lemma~\ref{lem:biLi}.
The main remaining difficulty is the possibility that $y$ exhibits self-contact on the boundary. 
To handle this, we  proceed as in the proof of Theorem~\ref{thm:convergenceLE} $(ii)_1$. 
First  create a small gap around the boundary, 
using $y\circ \Psi_j$ with smooth injective maps $\Psi_j:\bar\Omega\to \Omega$ close to the identity from Lemma~\ref{lem:shrink} that shrink the Lipschitz domain $\Omega$ into itself.  The composition  $y\circ \Psi_j(\overline{\Omega})$ is compactly contained in the open set $y(\Omega)$, and $y\circ \Psi_j$ is still locally bi-Lipschitz.
Moreover, $y\circ \Psi_j\to y$ in $W^{2,s}\cap W^{1,p}$ by Lemma~\ref{lem:shrinkcont}.
Further approximations of $y\circ \Psi_j$  with finite elements in $Y_h$ based on \eqref{fespaceapprox} 
can now be made while maintaining a safe distance from self-contact or loss of local invertibility. 
In particular, since 
$\liminf_j \inf \det\nabla (y\circ \Psi_j)>0$,
we avoid the singularity of $W$ and $E^{el}$ behaves continuously along our sequence  as in (i).  
 We  obtain $y_{j,k}\in Y_{h(k)}$ such that as $k\to\infty$, $y_{j,k}\to y\circ \Psi_j$ in $W^{2,s}$, $E_{\epsp(k)}^{\partial\Omega}(y_{j,k})=0$ for all sufficiently large $k$ (using Corollary~\ref{cor:whenECNvanishes}) and $E_{\eps(k),\sigma}^{h(k)}(y_{j,k})\to E_\sigma(y\circ \Psi_j)$. 
A suitable diagonal sequence $y_k:=y_{j(k),k} \to y $ with $j(k)\to\infty$ (slow enough)
 now yields the assertion. 
\end{proof}

\subsection{Remarks on the theoretical results}\label{ssec:remarks}

\begin{rem}[Generalizations: force terms and boundary conditions] \label{rem:DC}
Both Theorem~\ref{thm:convergence} and Theorem~\ref{thm:convergenceLE} can be easily generalized by adding a term to energy which is weakly lower semicontinuous and strongly continuous in the relevant space, i.e., in $W^{2,s}$ for the former and in $W^{1,2}$ for the latter. 
This comprises potentials associated to many typical force terms, including those used in our numerical experiments of Section~\ref{sec:num}. Generalization including boundary conditions, say, a Dirichlet condition like
$y=\identity$ on a closed subset $\Gamma_D$ of $\partial\Omega$, would also make sense, 
as long as the boundary condition is compatible with the constraints and stays away from self-contact. 
The proofs of the theorems can be extended to cover this case: While the lower bound (i) in Theorem~\ref{thm:convergence} and Theorem~\ref{thm:convergenceLE} is not affected at all, constructions for (ii) have to be adjusted to respect additional boundary conditions, but
Lemma~\ref{lem:shrink} with $\Gamma:=\Gamma_D$ is suitable for this purpose. 
However, the second and bigger problem is hidden in the assumption \eqref{fespaceapproxLE}, i.e., 
the density of suitable finite elements, now with a Dirichlet condition on $\Gamma_D$ added to the definition of $\cY$.
It is not clear if Proposition~\ref{prop:confFE} can be extended to this case.
\end{rem}

\begin{rem}[Pure traction problems]\label{rem:puretraction}
In our numerical experiments, we do impose a Dirichlet condition on a part of the boundary to avoid problems with coercivity.
If such boundary conditions are completely dropped, this  
leads to so-called pure traction problems, where the deformation is nowhere fixed but subject to additional (conservative) body and surface forces. The rigorous asymptotic derivation of linear from nonlinear elasticity is subtle in such a scenario 
\cite{MaPeTo19a,MaPeTo19b,MaMo21a,MaiPe21a}: The forces have to be suitably equilibrated to avoid energies which are not even bounded from below along rigid translations, and it is not always obvious which rigid motion is preferred by the forces as the natural point to linearize at. 

The approaches of \cite{MaPeTo19a,MaPeTo19b,MaiPe21a} and \cite{MaMo21a}, respectively, 
differ in the way the displacement is defined from a given deformation map $y$.
If we assume for simplicity that we linearize at the identity, then apart from rescaling for the small strain limit, the options are either the standard displacement $u=y-\identity$ \cite{MaPeTo19a,MaPeTo19b,MaMo21a} or 
the renormalized displacement $\tilde{u}=R_y^\top(y-(R_yx+c_y))$ with respect to the $y$-dependent optimal ``reference configuration'' $R_yx+c_y$, the rigid motion which minimizes $\tilde{u}$ in $W^{1,2}$ \cite[(1.3)]{MaMo21a}. The latter can avoid the technical condition of ``compatibility'' (of forces) \cite[(2.25)]{MaPeTo19a} in context of compactness (see also \cite{MaiPe21a} for a deeper discussion). 
In addition, both approaches suggest extending the linear elastic limit model, 
by introducing admissible states consisting of a pair $(u,W_0)$, the displacement and an antisymmetric matrix $W_0$
which represents an infinitesimal rotation and enters the energy as a correction. However, global minimizers in the limit models typically 
can be obtained with $W_0=0$, see \cite[Corollary 4.2]{MaPeTo19a} 
and \cite[p.5]{MaMo21a}, respectively.

For our purposes in Subsection~\ref{ssec:LE}, this means we have to be careful how we should interpret our linear elastic model as an approximation of a nonlinear elastic pure traction problem with moderate strains. Clearly, we need that the latter was already properly rotated so that the optimal rotation to linearize at is given by the identity. (If compatibility does not hold, the optimal rotation is not uniquely determined by the forces and not visible in the linearized model!) Fortunately, using $u=y-\identity$ as before is reasonable even if we choose to follow the point of view of \cite{MaMo21a}, because this does match their construction of the recovery sequence (the proof of the ``upper bound'' of \cite[Thm. 5.2]{MaMo21a}) when approximating a global minimizer with $W_0=0$ (and the optimal rotation normalized to $R_0=\idmatrix$). 
Be warned that more general forces or additional constraints can potentially further complicate the picture.
\end{rem}

\begin{rem}[ $\Gamma$-convergence]~\label{rem:Gamma-convergence}
Combined, (i) and (ii) in Theorem~\ref{thm:convergence} are equivalent to Mosco convergence \cite{Mo69a} of $E_{\eps(k),\sigma}^{h(k)}$ to $E_{\sigma}$,
which is stronger than $\Gamma(W^{2,s}\text{-weak})$-convergence (see, e.g., \cite{Dal93B}), since it requires the existence of a \emph{strongly} converging \say{recovery sequence} in (ii). Among other things, the limit functional is always unique if it exists in the sense that (i) and (ii) hold.
\end{rem}

\begin{rem}[Convergence of discrete minimizers]
(i) and (ii) in Theorem~\ref{thm:convergence} imply that a sequence of minimizers of $E_{\eps(k),\sigma}^{h(k)}$ always has a subsequence which weakly converges to a minimizer of $E_{\sigma}$. In the linear elastic setting of Theorem~\ref{thm:convergenceLE}, this is very similar, except that
(ii)$_1$ and (ii)$_2$ combined still do not fully cover the borderline case $y\in \cY\setminus \hat\cY$ where the local constraint is active. Any $y^*\in \hat\cY$ that arises a weak limit in $W^{1,2}$ of a sequence $(y_k)$ of minimizers of $E_{\epsp(k)}^{h(k)}$ is automatically a minimizer of $E$ in 
$\cY$, but we do not know what happens if $y^*\in \cY\setminus \hat\cY$.
\end{rem}

\begin{rem}[Possible non-uniqueness of minimizers]
When we compare Theorem~\ref{thm:convergence} and Theorem~\ref{thm:convergenceLE} to classical numerical convergence results, (i) and (ii) in a sense play the role of stability and consistency, respectively. 
As explained in more detail in the previous remark, we only get (weak) convergence of a sequence of discrete minimizers up to a subsequence, though.
In our scenarios, much more cannot be expected because in general, global minimizers do not have to be unique, as for instance the classical example of the  buckling  rod shows in the nonlinear elastic setting. Moreover, even in the linear elastic setting, the nonlinear constraints can break uniqueness.
One such example was observed in \cite{FoFreRo08a} for a local determinant constraint. The global invertibility constraint \eqref{ciarletnecas} apparently can break uniqueness, too, for instance in our numerical pincer example, where, when pressed enough, the pincers naturally have two symmetric ways of sliding past each other. As far as we know, no analytical results on this kind of nonlocally driven bifurcation scenario are available so far, though.
\end{rem}

\begin{rem}[Further errors due to numerical integration]
As defined, $E^h_{\epsp,\sigma}$ and $E^h_{\epsp}$ are assumed to be exact on their associated finite element space. In practice, 
additional approximations are usually needed at this point. In order to not break the analysis above,
corresponding additional errors terms should converge to zero as $(h,\epsp)\to 0$ along any sequence of states with bounded penalized energies. 
As a rule of the thumb, this forces a scaling regime where the mesh is fine enough with respect to $\epsp$, i.e., $h<<\epsp$.
\end{rem}

\begin{rem}[Existence of penalized minimizers] \label{ex:min-pen}
By standard applications of the direct method, for fixed $\eps$ and $h$, we can always get the existence of minimizers for the penalized energies,
both for the nonlinear elastic
$E_{\epsp,\sigma}$, $E^h_{\epsp,\sigma}$ and the linear elastic $E_{\epsp}$, $E^h_{\epsp}$.
In particular, the penalty term $E_\epsp^{\partial\Omega}$ is continuous in $L^1(\partial\Omega;\RR^d)$, a space into which the trace embeds compactly in both of our models (for deformations with bounded energy). Moreover, the constraints of $E_{\epsp}$, $E^h_{\epsp}$ built into $\cY$ are stable under weak convergence in $W^{1,2}$ on sets of bounded energy, cf.~the proof of Theorem~\ref{thm:convergenceLE} (i). 
We also point out that without additional terms or boundary conditions, all energies are translation invariant:
constant vectors can be added to $y$ without changing the energy. In addition, the linear elastic energy is invariant with respect to addition of linear transformations with vanishing symmetric part of the matrix.
Nevertheless, coercivity of the energy can be recovered by working in subspaces of $W^{2,s}(\Omega;\RR^d)$ or $W^{1,2}(\Omega;\RR^d)$ that remove these symmetries, say, by fixing appropriate averages. This extra step is not needed if added  constraints, 
boundary conditions or terms in the energy already fix or control the otherwise free constants. 
\end{rem}

\subsection{Conforming finite elements for the constrained linear elasticity} 

We now show that asymptotically conforming finite elements 
as assumed in \eqref{fespaceapproxLE} for Theorem~\ref{thm:convergenceLE} actually exist.
\begin{prop}\label{prop:confFE}
Suppose that the bounded domain $\Omega\subset \RR^d$ has a polygonal 
boundary and thus can be triangulated\footnote{i.e., $\Omega=\bigcup_{k} \overline{T_k}$ with finitely many simplices $T_k$ with pairwise disjoint interior (triangles for $d=2$, tetrahedra for $d=3$)}. Moreover, let $H\subset (0,\infty)$ with $\inf H=0$ and let
$Y_h\subset W^{1,\infty}(\Omega;\RR^d)$, $h\in H$, such that each $Y_h$ is the set of all functions which are piecewise affine with respect to a simplicial mesh associated to $Y_h$, of mesh size at most $h$ and triangulating $\Omega$. Then \eqref{fespaceapproxLE} holds.
\end{prop}
\begin{proof}
To simplify notation, we will write the proof only for the case $l=\tilde{l}$ and $L=\tilde{L}$.
Let $y\in \cY(\varrho;l,L)$, fix $\tau>1$ and abbreviate
\[
	\text{$m:=\frac{l}{\tau}<l$, $M:=L\tau>L$.}
\]
We suffices  to show that for every $\eps>0$
and all $h\in H$ sufficiently small (depending on $\eps$ and $y$), there exists 
$y_h\in Y_h\cap \cY(\varrho;m,M)$ such that
\begin{align}\label{pcFE-0}
	\norm{y-y_h}_{W^{1,2}(\Omega;\RR^d)}<2\eps.
\end{align}
Here, recall that $\cY(\varrho;m,M)$ is the set of all $y\in W^{1,2}(\Omega;\RR^d)$ such that 
$\det\nabla y>0$ a.e. and \eqref{idangleconstraint} holds with $m,M$ instead of $l,L$. 
The latter requires that $y$ is $\varrho$-locally Lipschitz with constant $M$ and 
satisfies the angle condition
\begin{align}\label{pcFE-ac}
\bald
	(y(x_1)-y(x_2))\cdot (x_1-x_2)\geq m\abs{x_1-x_2}^2 & \\
	\quad\text{for all $x_1,x_2\in\Omega$ with $\abs{x_1-x_2}<\varrho$}&.
\eald
\end{align}
The approximating maps are constructed in two consecutive steps, mollification and interpolation. The former is a bit more subtle than usual, because we lack extension results that would preserve \eqref{pcFE-ac}. So we use domain shrinking instead to avoid troubles near the boundary. 

\noindent{\bf Step 1: Mollification.}\\
Choose a family of standard mollifying kernels $\varphi_r$, $r>0$, i.e.,
$\varphi_r\in C_c^\infty(B_r(0);[0,\infty))$, $\varphi_r(rz)=r^{-d}\varphi_1(z)$
with $\int \varphi_1=1$, let $*$ the convolution operator.
With the domain shrinking maps $\Psi_j$ of Lemma~\ref{lem:shrink} and Lemma~\ref{lem:shrinkcont},
$(\varphi_r*y)\circ \Psi_j$ is well defined on all of $\Omega$ for $r$ small enough, more precisely, $r<\Dist{\Psi_j(\Omega)}{\partial\Omega}$.

We define
\[
	\hat{y}=\hat{y}_{r,j}:=(\varphi_r*y)\circ \Psi_j\in C^2(\overline\Omega;\RR^d),
\]
We claim that for suitable $r$ (small enough) and $j$ (big enough),
we have that
\begin{align}\label{pcFE-1}
	\norm{y-\hat y}_{W^{1,2}(\Omega;\RR^d)}<\eps\quad\text{and}\quad
	\hat y\in\cY(\varrho;\hat{m},\hat{M}),
\end{align}
where 
\[
	\text{$\hat{\tau}:=1+\frac{\tau-1}{2}$, $\hat{m}:=\frac{l}{\hat{\tau}}$ and $\hat{M}:=L\hat{\tau}$.}
\]
Notice that $l>\hat{m}>m$ and $L<\hat{M}<M$.

The bound for $\norm{y-\hat y}_{W^{1,2}}$ in \eqref{pcFE-1} follow for large enough $j$ from the properties of $\Psi_j$ obtained in
Lemma~\ref{lem:shrink} and Lemma~\ref{lem:shrinkcont}. It remains to show that $\hat y\in\cY(\varrho;\hat{m},\hat{M})$,
i.e., that the upper and lower bounds of \eqref{idangleconstraint} hold for $\hat y$ with the constants $\hat{m},\hat{M}$ instead of $l,L$, and that
$\det\nabla \hat y>0$ a.e..

{\bf Upper bound of \eqref{idangleconstraint}:} Since $y$ is $\varrho$-locally $L$-Lipschitz and $\hat{M}>L$, 
$\hat{y}$ is $\varrho$-locally $\tilde{M}$-Lipschitz as long as $r$ and $j$ are small and big enough, respectively. 

{\bf Lower bound of \eqref{idangleconstraint}:} We have to show that
\begin{align}\label{pcFE-ac2}
\bald
	(\hat{y}(x_1)-\hat{y}(x_2))\cdot (x_1-x_2)\geq \hat{m}\abs{x_1-x_2}^2&\\
	\text{for all $x_1,x_2\in\Omega$ with $\abs{x_1-x_2}<\varrho$}&.
\eald
\end{align}
In view of the definition of $\hat{y}$, this amounts to
\begin{align}\label{pcFE-1-0}
\bald{}
	[(\varphi_r * y)(\Psi_j(x_1))-(\varphi_r * y)(\Psi_j(x_2))]
	\cdot (x_1-x_2)
	\geq \hat{m}\abs{\xi_1-\xi_2}^2.
\eald
\end{align}
for $x_1,x_2\in \Omega$ with $\abs{x_1-x_2}<\varrho$.
For every $\xi_1,\xi_2\in \Omega$ with $\dist{\xi_i}{\partial\Omega}\geq r$ and $\abs{\xi_1-\xi_2}<\varrho$,
\eqref{pcFE-ac} implies that
\begin{align}\label{pcFE-1-1}
\bald
	&[(\varphi_r*y)(\xi_1)-(\varphi_r*y)(\xi_2)]\cdot (\xi_1-\xi_2)\\
	&=\int_{B_r(0)} \varphi_r(z)\big(y(z+\xi_1)-y(z+\xi_2)\big)\cdot 
	(\xi_1+z-(\xi_2+z))\, \dz\\
	&\geq l\int_{B_r(0)}\varphi_r(z) \abs{\xi_1+z-(\xi_2+z)}^2\\
	&= l\abs{\xi_1-\xi_2}^2 \, \dz
\eald
\end{align}
We now set $\xi_i=\xi_i(j):=\Psi_j(x_i)$, $i=1,2$, in \eqref{pcFE-1-1}. In addition,
we can replace all differences $\xi_1-\xi_2$ occurring in \eqref{pcFE-1-1}
by $x_1-x_2$ with small enough error to obtain \eqref{pcFE-1-0}. Here, the gap between $l$ and the smaller $\hat{m}$ can be used to absorb the error
for big enough $j$,
since
\begin{align}\label{pcFE-1-2}
\abs{\xi_1(j)-\xi_2(j)-(x_1-x_2)}\leq \operatorname{Lip}(\Psi_j-\identity)\abs{x_1-x_2}. 
\end{align}
Above, $\operatorname{Lip}(\Psi_j-\identity)$ denotes the global Lipschitz constant of $\Psi_j-\identity$ which converges to zero as $j\to\infty$.

{\bf $\mbf{\det\nabla \hat y>0}$ a.e.:} We know that $\det\nabla \hat y\geq \hat{m}^d>0$ in $\Omega$. 
as a consequence of \eqref{pcFE-ac2}.

\noindent{\bf Step 2: Piecewise affine interpolation.}\\
For a simplex $T=\operatorname{co}\{z_i\mid i=0,\ldots,d\}\subset \RR^d$ and a function $f:\overline{\Omega}\to\RR^d$, we define the affine interpolation $I[f]$ of $f$ on $T$ as the unique affine function coinciding with $f$ on all $(d+1)$ corners $z_i$ of $T$.
For any given triangulation of $\Omega$ into simplices, $I[f]$ is defined piecewise on each simplex of the triangulation,
which gives a continuous, piecewise affine function on $\overline{\Omega}$.

Since the function $\hat{y}$ obtained in the previous step is of class $C^2$ up to the boundary
and $\diam T\leq h$ for each simplex $T$, it is easy to see that 
\begin{align}\label{paff-approx-error}
	\norm{I_h[\hat{y}]-\hat{y}}_{C^1}\leq h \normn{\hat{y}}_{C^2}\underset{h\to 0}{\To} 0,
\end{align}
where $I_h$ denotes the piecewise affine interpolation with respect to the triangulation associated to $Y_h$.

In addition, $I_h$ preserves the local Lipschitz constant of $\hat{y}$, possibly up to a small change of the constant: 
For $x_1,x_2\inf \Omega$ with $\abs{x_1-x_2}<\varrho$, choose simplices $T_i\ni x_i$ of the triangulation of $\Omega$ associated to $Y_h$.
If $\diam(T_1\cup T_2)<\varrho$, the bound 
\[
	\abs{I_h[\hat{y}](x_1)-I_h[\hat{y}](x_2)}\leq (\hat{M}+h)\abs{x_1-x_2}
\]
is inherited from $\hat{y}$ using \eqref{paff-approx-error}. Otherwise,
there exists at least two nodes $z_1,z_2$ of the grid (corners of some simplex belonging to the triangulation, so that
$I_h[\hat{y}](z_i)=\hat{y}(z_i)$) such that
$\abs{z_1-z_2}<\varrho$ and
$\abs{z_i-x_i}\leq 2h$, $i=1,2$. Assuming that $2h<\varrho$, by the triangle inequality 
we see that
\[
\bald
	&\abs{I_h[\hat{y}](x_1)-I_h[\hat{y}](x_1)} \\
	&\leq \abs{I_h[\hat{y}](x_1)-I_h[\hat{y}](z_1)}+\abs{I_h[\hat{y}](x_2)-I_h[\hat{y}](z_2)}+
	\abs{\hat{y}(z_1)-\hat{y}(z_2)} \\
	&\leq 4\hat{M}h+\hat{M}\abs{x_1-x_2}\leq \Big(\frac{4\hat{M}h}{\varrho-2h}+\hat{M}\Big)\abs{x_1-x_2}.
\eald
\]
Here, for the last inequality, we used that this scenario can only occur if $\abs{x_1-x_2}\geq \varrho-2h$.

Next, we show that \eqref{pcFE-ac} holds for $I_h[\hat{y}]$ instead of $y$.
Let $x_1,x_2\in\Omega$ with $\abs{x_1-x_2}<\varrho$, 
contained in simplices $T_1,T_2$ of the triangulation of $Y_h$.
By \eqref{paff-approx-error}, we in particular have that
\begin{align}\label{pcFE-2-1}
	\norm{\nabla I_h[\hat{y}]-\nabla\hat{y}}_{L^\infty}\leq h\norm{D^2 \hat{y}}_{L^\infty},
\end{align}
and consequently, 
by the mean value theorem,
\begin{align}\label{pcFE-2-2}
	|I_h[\hat{y}](x_1)-I_h[\hat{y}](x_2)-(\hat{y}(x_1)-\hat{y}(x_2))|\leq 2 h |x_1-x_2| \norm{D^2 \hat{y}}_{L^\infty}.
\end{align}
With \eqref{pcFE-2-2}, \eqref{pcFE-ac} for $I_h[\hat{y}]$ thus follows from \eqref{pcFE-ac2} if $h$ is small enough.
Finally, \eqref{pcFE-ac} for $I_h[\hat{y}]$ implies a positive sign of $\det\nabla I_h[\hat{y}]$ as in Step 1.
\end{proof}

\section{Numerical experiments\label{sec:num}}

\subsection{A few useful explicit formulas}

\begin{ex}[A test case to check implementations of $E^{\partial\Omega}_{\epsp}$]
We replace $\partial\Omega$ by two line segments in $\RR^2$ (no longer the boundary of a domain, but this is irrelevant for the computation):
\[
	"\partial\Omega":=([0,1]\times \{1\})\cup ([0,1]\times \{0\})\subset \RR^2,
\]
Two parameters $a,b\geq 0$ determine the deformation $y$ we are interested in: $a$ causes a horizontal shift of the upper line segment, $b$ the distance of the lines after deformation.
For $x=(x_1,x_2)\in "\partial\Omega"$, we set
\[
	y(x)=(y_1(x),y_2(x)):=(x_1,bx_2)+\left\{\begin{array}{ll} (a,0) \quad & \text{if $x_2=1$},\\
	(0,0) \quad & \text{if $x_2=0$.}
	\end{array}\right.
\]
For $\epsp<1$, 
$E^{\partial\Omega}_{\epsp}(y)$ can now be calculated as follows:
\[
\bald
	\epsp^{\beta+d-1}E^{\partial\Omega}_{\epsp}(y)
	&=\int_0^1 \int_0^1 H_{\epsp}(t,t-s) \, \ds\, \dt \\
	&=\int_0^1 \int_{\max\{t-1,\gamma-\delta\}}^{\min\{t,\gamma+\delta\}} H_{\epsp}(t,r) \, \dr\, \dt.
\eald
\]
Here, we used the integrand
\[
	H_{\epsp}(t,r):=P\left(g\big(\sqrt{r^2+1}\big)-g\Big(\frac{1}{\epsp}\sqrt{b^2+(r-a)^2}\Big)\right)
\]
and the associated constants
\[
	\gamma:=\frac{a}{\sqrt{1-\epsp^2}},\quad \delta:=\left\{\baldat{2}
		& \sqrt{\frac{\epsp^2-b^2}{1-\epsp^2}} \quad &&\text{if $\epsp\geq b$,}\\
		& 0 \quad &&\text{else.}
	\ealdat\right.
\] 
Notice that $\gamma$ and $\delta$ were defined in such a way that
\[ 
	\text{$H_{\epsp}(t,r)>0\quad$
	if and only if $\quad \gamma-\delta<r<\gamma+\delta$},
\]
(provided that $P(\tau)>0$ if and only if $\tau>0$, 
in addition to \eqref{sCN1b} and \eqref{sCN1c}). For instance, in the special case $P(\cdot)=[\cdot]^+$ and
$g=\identity$, we get that
\[
\bald
	\epsp^{\beta+d-1}E^{\partial\Omega}_{\epsp}(y)
	&=\int_0^1 \int_{\max\{t-1,\gamma-\delta\}}^{\min\{t,\gamma+\delta\}}
	\Big(\sqrt{r^2+1}-\frac{1}{\epsp}\sqrt{b^2+(r-a)^2}\Big)
	\, \dr\, \dt,
\eald
\]
which now an be evaluated using standard software like Mathematica.
\end{ex}

\subsection{Simulations in 3D for linear elasticity}
We are motivated by 2D simulations of \cite{KroeVa19a} assuming the bulk version of penalization \eqref{sCN1} and perform 3D energy minimization evaluations with surface penalty \eqref{sCN2}.  Therefore, we consider $d=3$ and the approximate deformation $$y = (y_1, y_2, y_3) \in W^{1,s}(\Omega;\RR^3)$$ is searched for as the (ideally global) minimizer of the 
\begin{align}
	E_{\epsp,\mu}(y)=\Ele(y) - E^{body}(y)
	+\mu_{\partial \Omega} E^{\partial \Omega}_{\epsp}(y)
	 \label{energy:total}
\end{align}
over a finite-dimensional space. We stay within the linear elastic model of Subsection~\ref{ssec:LE}, with the elastic energy \eqref{energy_elastic}.
The local constraint \eqref{idangleconstraint} will not be enforced during the computation. 
In fact, with the forces and boundary conditions we use, there is no real incentive to violate \eqref{idangleconstraint} for, say, $l=\frac{1}{2}$ and $L=2$ (and some small $\varrho>0$). In any case, \eqref{idangleconstraint} could still be checked a posteriori.

The three components of the deformation $y_1, y_2, y_3$ are discretized using the finite element method in the space of $P_1$ tetrahedral elements, i.e., linear and globally continuous functions. 

\subsubsection{Implementation details}
We assume the following sets of nodes of the tetrahedral mesh: 
\begin{itemize}[itemsep=2pt, topsep=2pt, partopsep=0pt,leftmargin=10pt]
\item[] $\mathcal{N}$ - the set of all nodes, 
\item[] $\mathcal{N}_{D}$ - the set of nodes corresponding to the Dirichlet boundary condition,
\item[] $\mathcal{N}_{NP}$ - the subset of surface nodes expected to contribute to the nonpenetration penalty term $E^{\partial \Omega}_{\epsp}(y)$. 
\end{itemize}
The set $\mathcal{N}_{NP}$ is defined a priori in computations and it holds 
$$|\mathcal{N}_{NP}| \ll |\mathcal{N}|.$$ 
The assumption of linear elasticity density \eqref{density} allows an efficient evaluation of energies $\Ele(y) - E^{body}(y) $. Instead of displacement $y$, we work with displacement 
$u(x)=y(x)-x, x \in \Omega$ approximated by the finite element method as 
\[ 
u(x) \approx u_{h}(x) = \sum_{j=1}^{3|\mathcal{N}|} \left(\mathbf{u}\right)_j \phi_j(x)
\]
with a vector $\mathbf{u} \in \mathbb{R}^{3|\mathcal{N}|}$ of degrees of freedom and finite element basis functions $\phi_j, j=1,\dots, 3|\mathcal{N}|$. Then we have approximated
\begin{equation}\label{quadratic_energy}
 \Ele(y) - E^{body}(y) \approx  \frac{1}{2} \,\mathbf{u}^T K \mathbf{u} - \mathbf{b}^T \mathbf{u},
\end{equation}
where
$K \in \mathbb{R}^{3|\mathcal{N}| \times 3|\mathcal{N}|}$ is a stiffness matrix and $\mathbf{b} \in \mathbb{R}^{3|\mathcal{N}|}$ is a loading vector. The vector $\mathbf{u}$ is further decomposed into three disjoint parts
$ \mathbf{u} = (\mathbf{u}_{D}, \mathbf{u}_{NP}, \mathbf{u}_{R}),$ where 
\begin{itemize}[itemsep=2pt, topsep=2pt, partopsep=0pt,leftmargin=10pt]
\item[] $\mathbf{u}_{D}$ is the vector of (prescribed) displacements in nodes $\mathcal{N}_{D}$,
\item[] $\mathbf{u}_{NP}$ is the vector of displacements in nodes $\mathcal{N}_{NP}$,
\item[] $\mathbf{u}_{R}$ is the vector of displacements in the remaining nodes. 
\end{itemize}
If $\mathbf{u}_{D}=0$ 
, the optimality conditions for $\frac{1}{2} \,\mathbf{u}^T K \mathbf{u} - \mathbf{b}^T \mathbf{u} \rightarrow \mbox{min }$yield a relation between the vectors $\mathbf{u}_{NP}$ and $\mathbf{u}_{R}$ in the form of the linear system of equations 
\begin{equation}
\begin{pmatrix}
K_{NP,NP} & K_{NP,R}\\
K_{R, NP} & K_{R,R}
\end{pmatrix}    
\begin{pmatrix}
\mathbf{u}_{NP} \\
\mathbf{u}_{R} 
\end{pmatrix}    =
\begin{pmatrix}
\mathbf{b}_{NP} \\
\mathbf{b}_{R} 
\end{pmatrix}
\end{equation}
with $K$ and $\mathbf{b}$ decomposed to its parts given above. 
Then the well-known Schur complement technique of linear algebra provides 
\begin{equation}\label{schur}
\mathbf{u}_{R} = K_{R,R}^{-1} (\mathbf{b}_{R} - K_{R, NP} \, \mathbf{u}_{NP}).    
\end{equation}
The relation \eqref{schur}, together with the Dirichlet condition $\mathbf{u}_{D}=0$, allows us to express the quadratic energy \eqref{quadratic_energy} in reduced form featuring the vector $\mathbf{u}_{NP}$ as the only argument:
\begin{equation}\label{energy_np}
  \hat E(\mathbf{u}_{NP})=\frac{1}{2} \,\mathbf{u}_{NP}^T S \mathbf{u}_{NP} - 
  \mathbf{\hat b}^T \mathbf{u}_{NP}-\hat c,
\end{equation}
with the quadratic, linear and constant contributions given by
\[
\begin{alignedat}{2}
    &S&&:=
    K_{NP,NP}-K_{R, NP}^T K_{R,R}^{-1}K_{R,NP},\\
    &\mathbf{\hat b}^T&&:=\mathbf{b}_{NP}^T
    - \mathbf{b}_{R}^T K_{R,R}^{-1} K_{R, NP}
    +\mathbf{b}_{R}^T K_{R, NP},\\
    &\hat c&&:=
    \frac{1}{2}\mathbf{b}_{R}^T K_{R,R}^{-1} \mathbf{b}_{R}.
\end{alignedat}
\]
Here, we used that $K_{R, NP}=K_{NP,R}^T$ and $K_{R,R}^{-1}=K_{R,R}^{-T}$ due to the symmetry of $K$. When looking for minimizers or critical points, the constant contribution $\hat{c}$ can of course be ignored.

As long as $\mathcal{N}_{NP}$ has been chosen well,
all surface nodes outside of $\mathcal{N}_{NP}$ remain far from any self-contact and
thus contribute nothing to the nonpenetration penalty term,
so that $E^{\partial \Omega}_{\epsp}(y)$ also only depends $\mathbf{u}_{NP}$, say,
\[
   E^{\partial \Omega}_{\epsp}(y)= \hat{E}^{\partial \Omega}_{\epsp}(\mathbf{u}_{NP}).
\]
The total energy \eqref{energy:total} then can be likewise expressed as a function of $\mathbf{u}_{NP} \in \mathbb{R}^{3|\mathcal{N}_{NP}|}$ only. This reduces the number of degrees of freedom from $3(|\mathcal{N}|-|\mathcal{N}_{D}|)$   to $3|\mathcal{N}_{NP}|$ at the cost of precomputations of the terms in \eqref{energy_np}. 

Besides this reduction, the matrix $S$ appearing in the leading part of the reduced elastic and potential energy $\hat E$ offers us an easy way of preconditioning the problem. To this end,
we use the following transformation of the reduced total energy in the computations:
\[
    E_{\text{precond}}(\mathbf{v}_{NP}):=\hat E\big(S^{-\frac{1}{2}}\mathbf{v}_{NP}\big)
    + \hat{E}^{\partial \Omega}_{\epsp}\big(S^{-\frac{1}{2}}\mathbf{v}_{NP}\big).
\]
Thus we have replaced the original variable $\mathbf{u}_{NP}$ by $\mathbf{v}_{NP}=S^{\frac{1}{2}}\mathbf{u}_{NP}$. Here, recall that as a symmetric positive definite matrix (corresponding to the coercivity of the linear elastic energy), $S$ has a well-defined square root that for instance can be computed by Cholesky decomposition. 
\begin{rem}
Heuristically, our preconditioning means
that we choose $L^2$ as the natural energy space for the piecewise affine boundary elements
represented by the variable $\mathbf{v}_{NP}$ in our computations. By contrast, the natural space for (piecewise affine interpolations of) $\mathbf{u}_{NP}$ would be the trace space $H^{\frac{1}{2}}$ on the boundary piece of $\Omega$ corresponding to the nodes in $\mathcal{N}_{NP}$. Here, notice that formally, $S$ is a discretization of a symmetric and invertible nonlocal pseudo-differential operator of order $1$
defining a continuous and coercive bilinear form on $H^{\frac{1}{2}}$.
Its condition number and largest eigenvalue are expected to scale like $h^{-1}$ with respect to the grid size $h$.
\end{rem}
\begin{rem}
The interaction of the preconditioning with the penalty energy $\hat{E}^{\partial \Omega}_{\epsp}$ is not so obvious. In that regard, our choice is vaguely inspired by the tests conducted in \cite[Fig.~10]{YuBraSchuKee21a}
(for a quite different nonlocal interaction requiring higher regularity, though), where 
the use of $H^{s}$-gradient flows turned out to be favorable. In our notation, this corresponds to preconditioning by a (discretized) fractional Laplacian of order $\frac{s}{2}$.
However, unlike in \cite{YuBraSchuKee21a}, our nonlocal penalty term is well-defined even in $L^2$ and so does not suggest a natural choice for $s$. For this reason,
we proceeded with the precondition best matching the \say{leading} term with respect to local regularity, i.e., the linear elasticity energy. This suggests the
use of the square root of the matrix $S$. As its inverse has a smoothening effect, it is reasonable to expect that it stabilizes the contributions of $\hat{E}^{\partial \Omega}_{\epsp}$ as well. 
\end{rem}

\subsubsection{Computational benchmark}
We assume a pincer-shaped domain $\Omega = \Omega_1 \cup \Omega_2 \cup \Omega_3 $   
which consists of block subdomains
\begin{eqnarray*}
&&\Omega_1 =  (0,6) \times (0,0.5) \times (2.5,3), \qquad \quad \, \, \,\mbox{(upper part)}\\
&&\Omega_2 = (0,0.5) \times (0,0.5) \times (0.5,2.5), \qquad \mbox{(middle part)} \\
&&\Omega_3 = (0,6) \times (0,0.5) \times (0,0.5), \qquad \quad \, \, \,\mbox{(lower part)}
\end{eqnarray*}

\begin{figure}[h]
\centering
\begin{subfigure}[b]{0.6\textwidth}
\centering
\includegraphics[width=\textwidth]{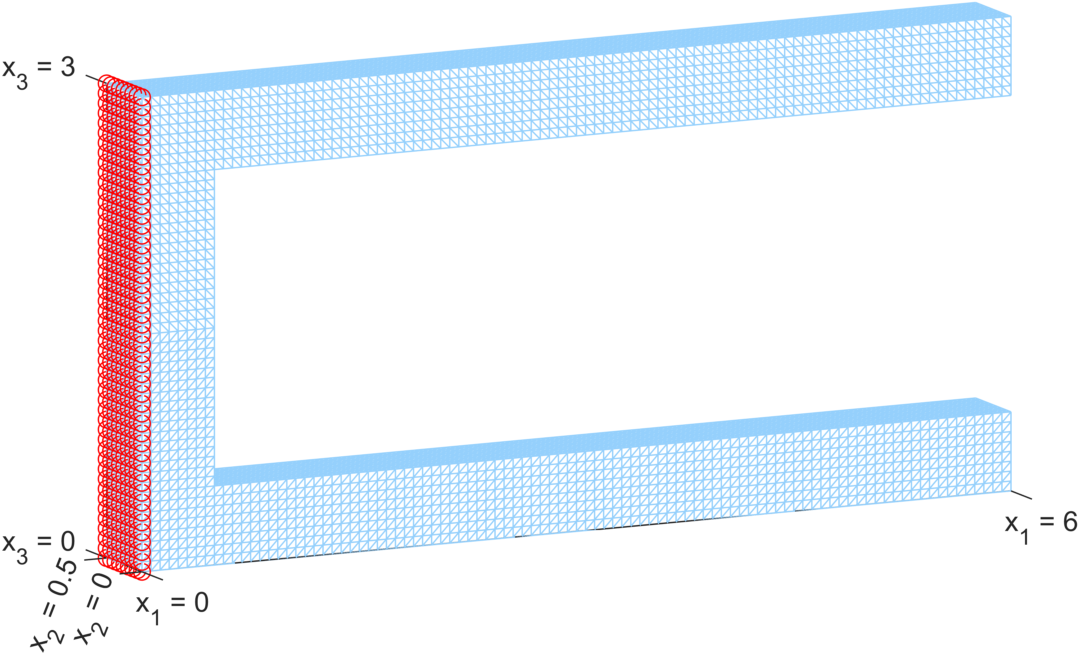} 
\caption{A nondeformed mesh. }
\label{fig:model_pincers3D}
\end{subfigure}
\centering
\begin{subfigure}[b]{0.6\textwidth}
\includegraphics[width=\textwidth]{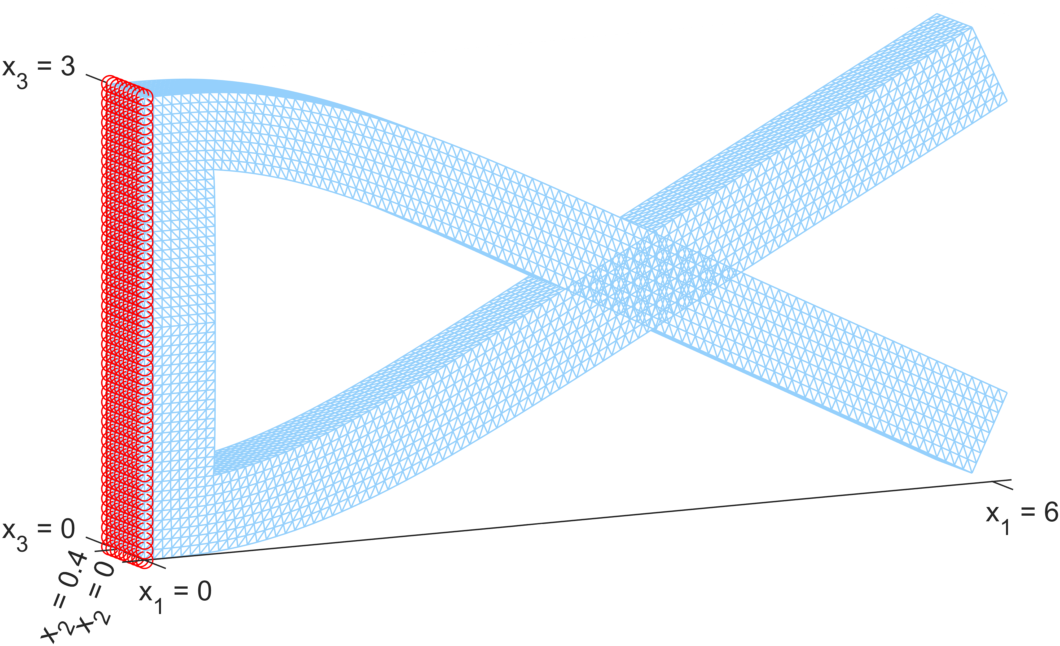} 
\caption{An elastically deformed mesh ($\mu=0$).
}
\label{fig:model_pincers3D_elasticity}
\end{subfigure}
\caption{Nondeformed and elastically deformed meshes (level 3) are discretized by tetrahedral elements. The red nodes indicate Dirichlet nodes $\mathcal{N}_{D}$ at which zero displacements are required.}
\end{figure}

The domain $\Omega$ is depicted in Figure \ref{fig:model_pincers3D} along with its tetrahedral triangulation and nodes corresponding to the homogeneous Dirichlet boundary conditions defined on a part of the domain boundary.
\begin{eqnarray*}
& y=0 \quad \mbox{for } x \in \Lambda_{D}= \{1 \} \times (0,0.5) \times (0.5,2.5).
\end{eqnarray*}
There is a linear body force term $E^{body}(y) $ with the energy contribution 
\[
	E^{body}(y):=\int_\Omega g_{\rm body}(x)\cdot (y(x)-x) \dx
\]
and a linear body force density
$$   g_{\rm body}(x_1, x_2, x_3) =g_{\rm load}  (0, 0, - H(x_1 - 2)H(x_1 - 4) \, \sign{(x_3 - 1.5)}) $$
on $(x_1, x_2, x_3) \in \Omega$,
where $H(\cdot)$ denotes the Heaviside step function, $\sign(\cdot)$ the signum function and $g_{\rm load}>0$ is a given loading parameter. This form of the linear body force density presses the tips of both pincer parts $\Omega_1$ and $\Omega_3$ against each other. \\

The surface penalty term $E^{\partial \Omega}_{\epsp}(y)$ is taken according to \eqref{sCN2} and assumes the choice 
\[ \beta=d-1+0.1=2.1, \qquad \epsp = s\,h/r, \]
where $s=3, r=2$ and $h$ is the grid size.
As to $P$ and $g$, we use $C^2$ functions with a fixed parameter $a:=0.01$ given by
\begin{align}\label{Pg-C2}
\begin{aligned}
    P(t)=g(t):=\left\{\begin{array}{ll}
        0 & \text{if $t<0$,}\\
        \frac{1}{a^2}t^3-\frac{1}{2a^3}t^4 & \text{if $0\leq t \leq a$,}\\
        t-\frac{a}{2} & \text{if $t>a$.}
    \end{array}\right.
\end{aligned}
\end{align}

For computational simulations, we consider the elastic material parameters - Young's modulus and Poisson's ratio
$$E = 2e8, \quad \nu = 0.3$$ 
 corresponding to the Lam\'e's parameters 
 $$\lambda = E \nu / (1+\nu)(1-2\nu) \approx 1.15e08, \quad \mu = E / (2(1+\nu)) \approx 7.69e07$$ 
 and the linear body force defined by $g_{\rm load}=4e05$. The solution of the purely linear elastic system with the nonpenetration penalty switched off (i.e., $\mu=0$) leads to the body interpenetration, see Figure \ref{fig:model_pincers3D_elasticity}. In order to prevent the interpenetration, the penalty term $\mu E^{\partial \Omega}_{\epsp}(y)$ must be switched on and we set $\mu_{\partial \Omega}=0.001 \,E$.

\begin{rem}
In our example, $r=2$ is precisely the distance of $\Omega_1$ and $\Omega_3$, the distance in reference configuration of the two inner pincer surfaces on opposing sides which will approach self-contact under the influence of the body force. As long as local deformations remain close to the identity, the above choice for $\epsp$ corresponds to an effective range of the penalty term of $s=3$ grid boxes on these contact surfaces. Here, recall that the penalty term is active at any pair of points $x_1,x_2$ iff $|y(x_1)-y(x_2)|<\epsp|x_1-x_2|$, 
and $|x_1-x_2| \approx r$ for the relevant material points where we expect self-contact in our example.
\end{rem}
\begin{rem}\label{rem:biLi-apost}
We do not enforce the local constraint \eqref{idangleconstraint} 
numerically,
and this also means that the local bi-Lischitz property
\eqref{biLi} and the local invertibility of deformations as required by our theoretical results are \emph{not} guaranteed. However, we can check \eqref{biLi} a posteriori.
In fact, it suffices to check the global behavior on the boundary and local invertibility on each element: For any $y\in W^{1,\infty}$ such that
$y$ is invertible on $\partial\Omega$ and $\det \nabla y>0$ a.e., 
$y$ is already a homeomorphism on $\overline{\Omega}$, as long as $\Omega$ is a Lipschitz domain without holes (i.e., $\RR^d\setminus \partial\Omega$ has only two connected components) \cite{Kroe20a}. 
If, in addition, $y$ is also continuous and piecewise affine (e.g.), \eqref{biLi} holds (with suitable constants) if and only if $\abs{\nabla y},\abs{(\nabla y)^{-1}}\in L^\infty$, where $\abs{\cdot}$ denotes the operator norm (modulus of the largest singular value, which never exceeded $1.2$ in our computations). Here, in the interior, one could actually choose $l^{-1}=\sup |\nabla y^{-1}|$, but smaller $l$ can still occur due to boundary effects if the deformation reduces angles from the outside between boundary elements. Notice that if $y$ is piecewise affine and invertible on the boundary, the formation of outer cusps in the deformed configuration 
(which would break \eqref{biLi}) is impossible.
\end{rem}

Depending on different initial deformations in the minimization of the total energy \eqref{energy:total} we discuss two numerical solutions.
The first is symmetric and probably corresponds to a local minimum. The second is asymmetric, and likely approximates one of two global minima related to each other by reflection.

\begin{figure} 
\centering
\begin{subfigure}[b]{0.8\textwidth}
\centering
\includegraphics[width=0.99\textwidth]{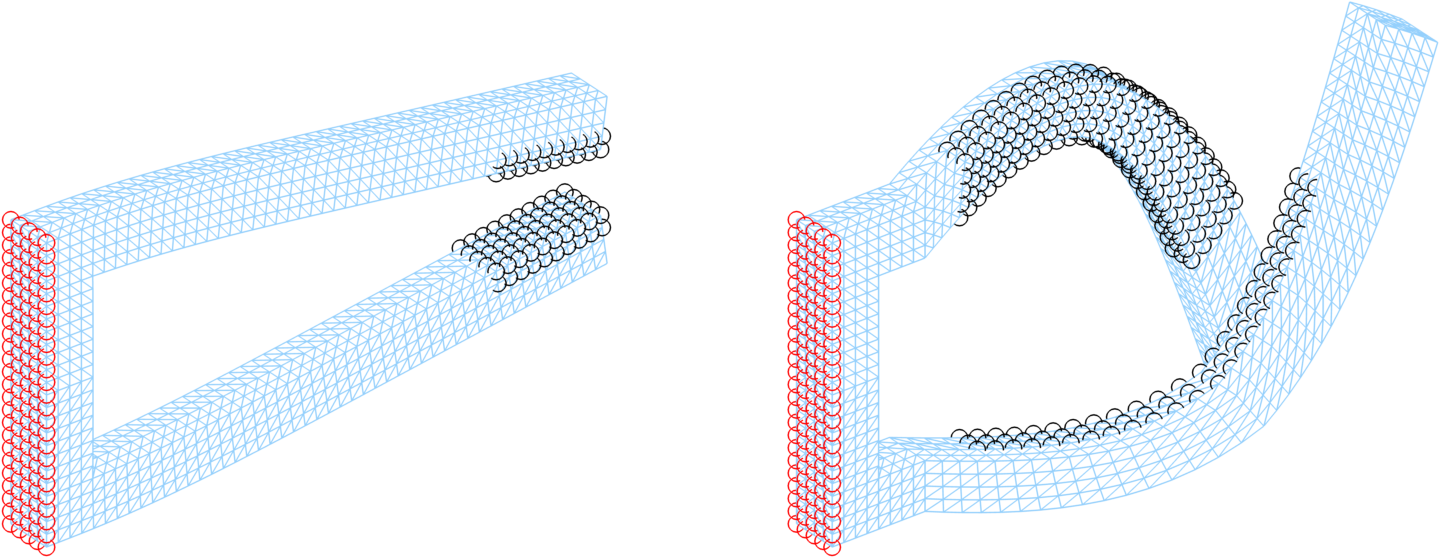}  
\caption{Initially deformed meshes.}
\label{fig:pincers-all-A}
\vspace{-0.5cm}
\end{subfigure}
\centering
\begin{subfigure}[b]{0.9\textwidth}
\includegraphics[width=0.99\textwidth]{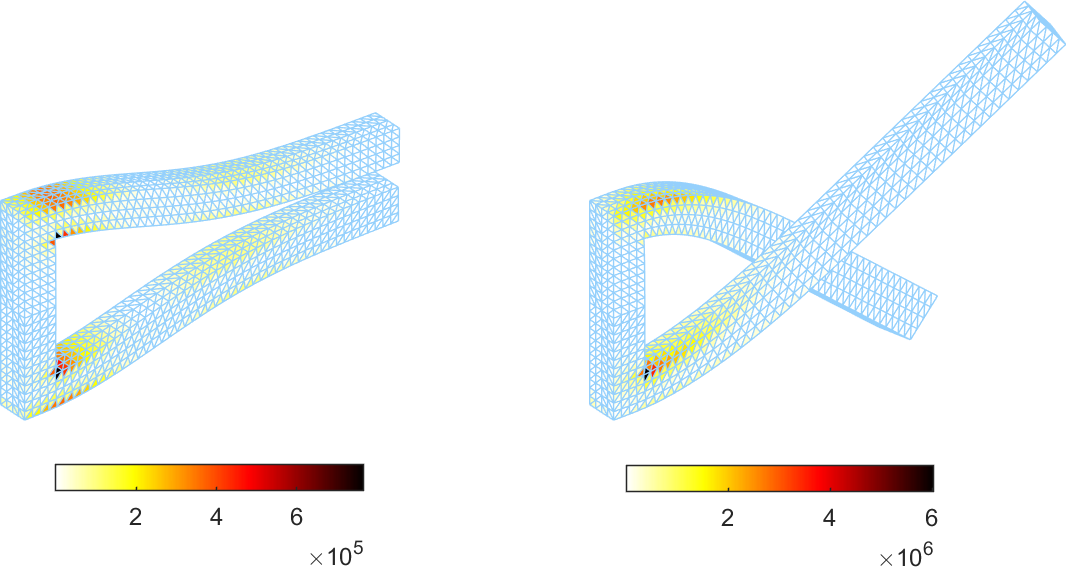}   
\caption{Deformed meshes with the underlying elastic densities.}
\label{fig:pincers-all-B}
\end{subfigure}
\centering
\begin{subfigure}[b]{0.9\textwidth}
\includegraphics[width=0.99\textwidth]{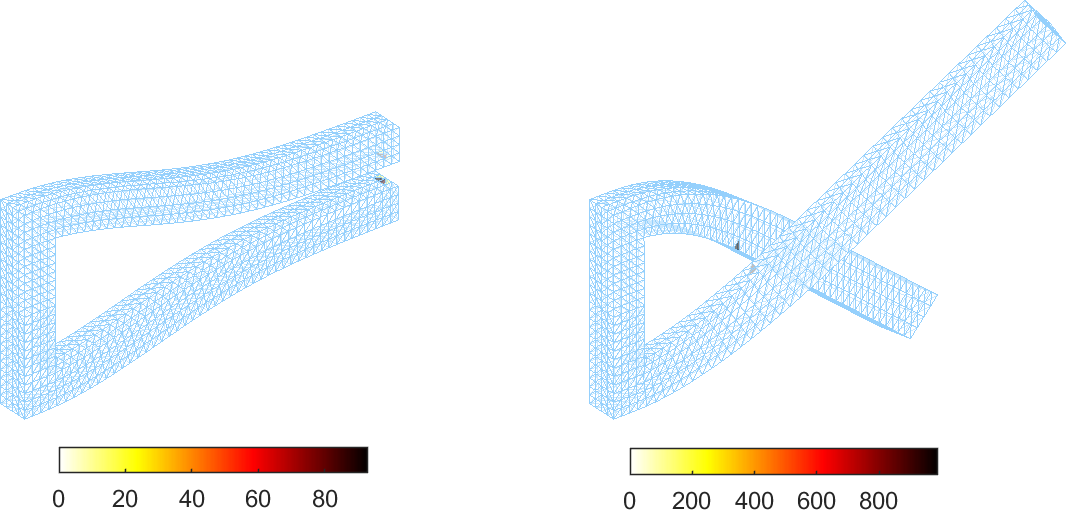}   
\caption{Deformed meshes with the underlying nonpenetration densities.
}
\label{fig:pincers-all-C}
\vspace{0.5cm}
\end{subfigure}
\begin{subfigure}[b]{0.9\textwidth}
\includegraphics[width=0.99\textwidth]{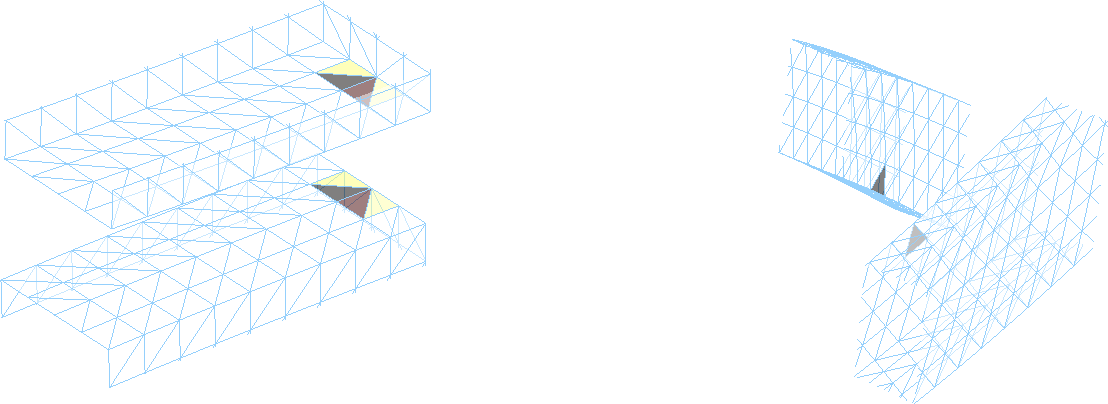}   
\caption{The above pictures zoomed in nonpenetration nodes.}
\label{fig:pincers-all-D}
\end{subfigure}
\caption{Solutions for level 2 mesh: 
the symmetric initial deformation (the left column) and the asymmetric initial deformation (the right column). }
\end{figure}

\subsubsection{Symmetric initial deformation}
The elastic deformation $y_{\rm elast}$ is evaluated by the solution of the purely linear elastic system and the initial deformation $y_{\rm init}$ is generated as
$$ y_{\rm init}(x) = x + 0.05 (y_{\rm elast}(x) - x), \qquad x \in \Omega.$$ The constant $0.05$ is chosen so that the initial deformation is out of self-contact, see the left part of Figure \ref{fig:pincers-all-A}.
The full minimization of  \eqref{energy:total} respecting the nonpenetration term $\mu=$ and taking $ y_{\rm init}(x)$ is the minimization procedure converges to the deformation displayed with linear elasticity and non-penetration density and its zoomed view in the left column of Figure
\ref{fig:pincers-all-B}, \ref{fig:pincers-all-C}, \ref{fig:pincers-all-D}. 
The detailed evaluation data is given in Table \ref{tab:symmetric}.

{\tiny
\begin{table}[b]
\centering
\begin{tabular}{ c|r | r| c | c | r | c | r |r } 
\multicolumn{3}{c|}{mesh}  &   \multicolumn{4}{c|}{energy} & \multicolumn{2}{c}{performance}\\
\hline
 lev. &  \multicolumn{2}{c|}{nodes:}  &total & elastic & nonpenet. & body & iters. & time \hspace{1em}  \\
 & $|\mathcal{N}|$ & $|\mathcal{N}_{NP}|$  & $E_{\epsp,\mu}(y)$ & $\Ele(y)$ & $\mu E^{\partial \Omega}_{\epsp}(y)$ & $E^{body}(y)$ & &(sec) \hspace{1em} \\ 
  \hline
1 & 513 & 0 \qquad & -3.88e+05 & 3.88e+05 &  0 \hspace{1em} & -7.76e+05 & 1 & 3.03e-03\\ 
 2 & 2825 & 0 \qquad & -5.38e+05 & 5.38e+05 &  0 \hspace{1em} & -1.08e+06 & 1 & 2.48e-02\\ 
 3 & 18225 & 0 \qquad & -6.03e+05 & 6.03e+05 &  0 \hspace{1em} & -1.21e+06 & 1 & 3.62e-01\\ 
 4 & 129761 & 0 \qquad & -6.19e+05 & 6.19e+05 &  0 \hspace{1em} & -1.24e+06 & 1 & 7.67e+00\\ 
\end{tabular}
\vspace{-0.2cm}
\caption{Purely elastic material.}
\label{tab:elastic}
\vspace{0.5cm}
\begin{tabular}{ c|r | r| c | c | r | c | r |r } 
\multicolumn{3}{c|}{mesh}  &   \multicolumn{4}{c|}{energy} & \multicolumn{2}{c}{performance}\\
\hline
 lev. &  \multicolumn{2}{c|}{nodes:}  &total & elastic & nonpenet. & body & iters. & time \hspace{1em}  \\
 & $|\mathcal{N}|$ & $|\mathcal{N}_{NP}|$  & $E_{\epsp,\mu}(y)$ & $\Ele(y)$ & $\mu E^{\partial \Omega}_{\epsp}(y)$ & $E^{body}(y)$ & &(sec) \hspace{1em} \\ 
  \hline
1 & 513 & 52 & -1.94e+05 & 9.16e+04 & 2.80e+01 & -2.86e+05 & 10 & 1.12e+00\\ 
 2 & 2825 & 146 & -2.52e+05 & 1.18e+05 & 7.42e+00 & -3.71e+05 & 18 & 1.98e+01\\ 
 3 & 18225 & 454 & -2.78e+05 & 1.28e+05 & 1.63e+00 & -4.06e+05 & 25 & 4.49e+02\\ 
\end{tabular}
\vspace{-0.2cm}
\caption{The initial symmetric deformation.}
\label{tab:symmetric}
\vspace{0.5cm}
\begin{tabular}{ c|r | r| c | c | r | c | r |r } 
\multicolumn{3}{c|}{mesh}  &   \multicolumn{4}{c|}{energy} & \multicolumn{2}{c}{performance}\\
\hline
 lev. &  \multicolumn{2}{c|}{nodes:}  &total & elastic & nonpenet. & body & iters. & time \hspace{1em}  \\
 & $|\mathcal{N}|$ & $|\mathcal{N}_{NP}|$  & $E_{\epsp,\mu}(y)$ & $\Ele(y)$ & $\mu E^{\partial \Omega}_{\epsp}(y)$ & $E^{body}(y)$ & &(sec) \hspace{1em} \\ 
  \hline
 1 & 513 & 130 & -2.70e+05 & 3.22e+05 & 3.79e+02 & -5.93e+05 & 14 & 6.99e+00\\ 
 2 & 2825 & 450 & -4.67e+05 & 4.86e+05 & 3.09e+01 & -9.52e+05 & 23 & 2.66e+02\\ 
\end{tabular}
\vspace{-0.2cm}
\caption{The initial asymmetric deformation.}
\label{tab:asymmetric}
\end{table}
}

\subsubsection{Asymmetric initial deformation}

Starting from a symmetric starting condition, the solver consistently terminated at a (almost) symmetric finale state where 
the two pincer ends are flatly pressed together. 
While this probably always is a local minimum, with sufficiently strong forces we expect to find a another candidate for the global minimum with less energy, 
a deformation where the pincer arms slide past each other. If this happens, there are obviously two such solutions that can be mapped into each other by reflection across the plane $\{x_2=0.25\}$, with the upper pincer (occupying $\Omega\cap \{x_3>1.5\}$ in its undeformed state) passing in front or in the back of the lower pincer.
To find such a deformation numerically, we artificially specify an explicitly defined starting deformation $y_0$ with the pincer arms passing each other. More specifically, we fix $X=(X_1,X_2,X_3):=(3,0.25,1.5)$, a point in the center between the two \say{arms} of $\Omega$, and introduce the (reflected) planar polar coordinates with respect to the center point $X$ given by
\[
\begin{aligned}
    &r=r(x_1,x_3):=\sqrt{(x_1-X_1)^2+(x_3-X_3)^2}>0,~~~\\
    &\theta=\theta(x_1,x_3):=-\sphericalangle\big((x_1-X_1,x_3-X_3),(-1,0)\big)\in (-\pi,\pi).
\end{aligned}
\]
Here, notice that the half-plane $X+\{(t,0,z)\mid t\geq 0,z\in\RR\}$ does not intersect $\Omega$. Consequently,
the angle $\theta$ is a well defined continuous 
extension of the function $-\arctan((x_3-X_3)/(x_1-X_1))$
from $\{x_1<X_1\}$ to $\Omega$. It satisfies $\theta(x_1,X_3)=0$ for all $x_1<X_1$, 
$\theta(x_1,x_3)<0$ for $x_3<X_3$ and $\theta(x_1,x_3)>0$ for $x_3>X_3$.
With the auxiliary \say{twist parameter} function
\[
    T=T(x_1):=\frac{1}{2}\min\{[x_1-X_1-0.5]^+,2\}
\]
we now define the starting deformation as
\[
    y_0(x):=\left(\begin{array}{r}
        -r \cos \big((1+0.2T)\theta\big)\\
        x_2+0.3T(x_3-X_3)\\
        r\sin \big((1+0.2T)\theta\big)
    \end{array}\right)
\]
A visualization of $y_0(\Omega)$ can be seen on the right in Figure~\ref{fig:pincers-all-A}. The optimal deformation displayed with linear elasticity and non-penetration density and its zoomed view on the right of Figure \ref{fig:pincers-all-B}, \ref{fig:pincers-all-C}, \ref{fig:pincers-all-D}. 
The detailed evaluation data is given in Table \ref{tab:asymmetric}. 

{\bf{Details on implementation and running times:}} Our MATLAB code is based on the FEM vectorization ideas of \cite{MoVa22a} combined with fast assembly routines of \cite{RaVa13a} for linear elasticity. Practical energy minimization is based on the first-order quasi-Newton method 
applied to $E_{\text{precond}}$, with an explicit differentiation behind and the construction of an approximate Hessian by the Broyden–Fletcher–Goldfarb–Shanno (BFGS) algorithm. The code is available at 
\begin{center}
\small
\url{https://www.mathworks.com/matlabcentral/fileexchange/124290} 
\end{center}
for download and testing. It requires running the Optimization Toolbox (function \say{fminunc}) and Statistics and Machine Learning Toolbox (function \say{pdist2}). Assembly times were obtained on
a MacBook Air (M1 processor, 2020) with 16 GB memory running MATLAB R2022a.

\section*{Acknowledgments}
 The authors are indebted to Alexej Moskovka for providing a 3D computational  mesh. They express their 
gratitude for the support and the Czech Science Foundation (GACR) grant 21-06569K \say{Scales and shapes in continuum thermomechanics}.

\bibliographystyle{plain}
\bibliography{NLEbib}

\end{document}